\documentclass[12pt]{article}
\usepackage{amsmath,amssymb, amsthm,graphicx, amscd,ifthen, amsfonts,wrapfig}
\usepackage[usenames]{color}
\usepackage{tikz}
\usetikzlibrary{patterns,decorations.text,positioning,arrows,shapes,decorations.markings}
\tikzstyle{vertex}=[circle,draw=black,fill=black,inner sep=0,minimum size=0.2cm,text=white,font=\footnotesize]

\newtheorem{theorem}{Theorem}
\newtheorem{lemma}{Lemma}
\newtheorem{corol}{Corollary}

\newtheorem{prb}{Problem}

\newcommand{\FO}{{\mathrm{FO}}}
\newcommand{\MSO}{{\mathrm{MSO}}}

\newcommand{\EHR}{{\mathrm{EHR}}}

\newcommand{\qd}{{\mathrm{q}}}

\title{Short monadic second order sentences about sparse random graphs}
\author{Andrey Kupavskii\thanks{Laboratory of Advanced Combinatorics and Network Applications, Moscow Institute of Physics and Technology; Emails: {\tt kupavskii@yandex.ru, zhukmax@gmail.com} \ \ Research supported by the grant RNF~16-11-10014.}\phantom{a\,}and Maksim Zhukovskii\footnotemark[1]}

\date{}

\begin{document}
\maketitle
\begin{abstract}
In this paper, we study zero-one laws for the Erd\H{o}s--R\'{e}nyi random graph model $G(n,p)$ in the case when $p = n^{-\alpha}$ for $\alpha>0$. For a given class $\mathcal{K}$ of logical sentences about graphs and a given function $p=p(n)$, we say that $G(n,p)$ obeys the zero-one law (w.r.t. the class $\mathcal{K}$) if each sentence $\varphi\in\mathcal{K}$ either a.a.s. true or a.a.s. false for $G(n,p)$. In this paper, we consider first order properties and monadic second order properties of bounded \textit{quantifier depth} $k$, that is, the length of the longest chain of nested quantifiers in the formula expressing the property. Zero-one laws for properties of quantifier depth $k$ we call the \textit{zero-one $k$-laws}.

The main results of this paper concern the zero-one $k$-laws for monadic second order properties (MSO properties). We determine all values $\alpha>0$, for which the zero-one $3$-law for MSO properties does not hold. We also show that, in contrast to the case of the $3$-law, there are infinitely many values of $\alpha$ for which the zero-one $4$-law for MSO properties does not hold. To this end, we analyze the evolution of certain properties of $G(n,p)$ that may be of independent interest.
\end{abstract}

\section{Introduction}\label{zero_one_intro}
In 1959, P. Erd\H os and A. R\' enyi, and independently E. Gilbert, introduced two closely related models for generating random graphs. A seminal paper of Erd\H os and R\'enyi \cite{Erdos}, that appeared one year later, brought a lot of attention to the subject, giving birth to the vast and ever-developing area of Erd\H os-R\'enyi random graphs. In spite of the name, the more popular model $G(n,p)$  is the one proposed by Gilbert. In this model, we have $G(n,p) = (V_n,E)$, where $V_n=\{1,\ldots,n\}$, and each pair of vertices is connected by an edge with probability $p$ and independently of other pairs. For more information, we refer readers to the books by B. Bollob\'as \cite{Bollobas} and S. Janson, T. \L uczak and A. Ruci\'{n}ski \cite{Janson}, entirely devoted to random graphs, as well as to the book of N. Alon and J. Spencer on probabilistic method \cite{AS}.

Studying zero-one laws requires some logical prerequisites. We review some of the basics in this paragraph, and refer the reader to~\cite{Logic2,Libkin,Strange,Veresh,Survey}. Formulae in the first order language of graphs (FO formulae) are constructed using relational symbols $\sim$ (interpreted as adjacency) and $=$, logical connectives $\neg,\rightarrow,\leftrightarrow,\vee,\wedge$, variables $x,y,x_1, \ldots$ that express vertices of a graph, quantifiers $\forall,\exists$ and parentheses $($, $)$.  Monadic second order, or MSO, formulae (see~\cite{Muller}, \cite{Tysk}) are built of the above symbols of the first order language, as well as the variables $X,Y,X_1,\ldots$ that are interpreted as unary predicates, i.e. {\it subsets} of the vertex set. Following \cite{Logic2}, \cite{Survey}, we call the number of nested quantifiers in a longest sequence of nested quantifiers of a formula $\varphi$ {\it the quantifier depth} $\qd(\varphi)$. Formulae must have finite length. For example, the MSO formula
\begin{equation}
 \forall X \quad\biggl(\biggl[\exists x_1\exists x_2\,\,
 (X(x_1)\wedge \neg X(x_2))\biggr]\rightarrow
 \biggl[\exists y \exists z\,\,(X(y)\wedge\neg X(z)\wedge y\sim z)\biggr]\biggr)
\label{connectedness_MSO}
\end{equation}
has quantifier depth $3$ and expresses the property of being connected. It is known that the property of being connected cannot be expressed by a FO formula (see, e.g.,~\cite{Survey}). The quantifier depth of a formula has the following algorithmic consequence: an FO formula of quantifier depth $k$ on an $n$-vertex graph can be verified in $O(n^k)$ time.


Many properties of graphs may be expressed via FO formulae. Somewhat surprisingly, Y. Glebskii, D. Kogan, M. Liogon'kii and V. Talanov in 1969, and independently R. Fagin in 1976, proved that \textit{any} FO formula is either asymptotically almost surely (a.a.s.) true or a.a.s. false for $G(n,1/2)$, as $n\to \infty$. In such a situation we say that $G(n,p)$ obeys the zero-one law for FO formulae.

More precisely, we say that  $G(n,p)$  {\it obeys the FO zero-one $k$-law} (resp. {\it the MSO zero-one $k$-law}) if any first order formula (resp. monadic second order formula) of quantifier depth $k$ is either a.a.s. true  or a.a.s. false  for $G(n,p)$. We say that $G(n,p)$ {\it obeys the FO zero-one law} ({\it the MSO zero-one law}) if it obeys the FO zero-one $k$-law (the MSO zero-one $k$-law) for any positive integer $k$.

In 1988 S.~Shelah and J.~Spencer~\cite{Shelah} proved the following zero-one law for the random graph $G(n,n^{-\alpha})$.
\begin{theorem}
Let $\alpha>0$. The random graph $G(n,n^{-\alpha})$ does not obey the FO zero-one law if and only if either $\alpha\in(0,1]\cap\mathbb{Q}$, or $\alpha=1+1/\ell$ for some integer $\ell$.
\label{FO01}
\end{theorem}
Obviously, there is no MSO zero-one law when even the FO zero-one law does not hold, so neither does the MSO zero-one law hold for rational $\alpha\in(0,1]$ nor for $\alpha=1+1/\ell$.  In 1993, J.~Tyszkiewicz~\cite{Tysk} proved that $G(n,n^{-\alpha})$ does not obey the MSO zero-one law for irrational $\alpha\in(0,1)$ as well. However, for the only remaining possibility $\alpha>1$ and $\alpha\neq 1+1/\ell$, the MSO zero-one law {\it does} hold. This follows via a standard argument in the theory of logical equivalence (see the detailed proof of this corollary in~\cite{Zhuk_Logic}, Theorem 2). In the next theorem, we summarize the known results concerning the MSO zero-one law for $G(n,n^{-\alpha})$.
\begin{theorem}
Let $\alpha>0$. The random graph $G(n,n^{-\alpha})$ does not obey the MSO zero-one law if and only if either $\alpha\in(0,1]$ or $\alpha=1+\frac 1{\ell}$ for some integer $\ell$.
\label{MSO01}
\end{theorem}

For a formula $\varphi$, we use the notation $G\models\varphi$ if $\varphi$ is true for $G$. {\it The spectrum} of $\varphi$ is the set $\mathcal{S}(\varphi)$ defined by
$$
(0,\infty)\setminus\mathcal{S}(\varphi):=\{\alpha>0:\,\lim_{n\to\infty}{\sf P}(G(n,n^{-\alpha})\models\varphi)\text{ exists and equals }0\text{ or }1\}.
$$
J. Spencer proved \cite{Spencer_inf} that there exists an FO formula with infinite spectrum. Moreover, M. Zhukovskii \cite{Zhuk_inf} constructed an FO formula with infinite spectrum and of quantifier depth $5$.\\

In this paper, we construct an MSO formula of quantifier depth $4$ and with infinite spectrum, and show that this quantifier depth is smallest possible.
\begin{theorem}
There exists an MSO formula $\varphi$ with $\qd(\varphi)=4$ and infinite $S(\varphi)$.
\label{inf_spectrum}
\end{theorem}
Moreover, we find \textit{all} values of $\alpha$ for which the MSO zero-one $3$-law does not hold.
\begin{theorem}
Let $\alpha>0$. The random graph $G(n,n^{-\alpha})$ does not obey the MSO zero-one $3$-law if and only if $\alpha\in\{\frac 45,\frac56,1,\frac98,\frac76,\frac65,\frac54,\frac43,\frac32,2\}$.
\label{3-law}
\end{theorem}

Note that $G(n,n^{-1-1/\ell})$ does not obey the MSO zero-one 3-law for all $\ell\in\{1,2,\ldots,8\}$ except for $\ell=7$.

In contrast to the case of the MSO zero-one $3$-law,  the random graph $G(n,n^{-\alpha})$ obeys the FO zero-one 3-law for all $\alpha\in(0,1)$~(see \cite{Zhuk_law}, Theorem 3). If $\alpha=1$, the FO zero-one 3-law fails (since the probability of being triangle-free tends to $e^{-1/6}$ --- see, e.g., Theorem~\ref{threshold} --- and `being triangle-free' is expressible in FO language with quantifier depth 3). By Theorem~\ref{FO01}, the zero-one $3$-law holds for all $\alpha>1$ such that $\alpha\neq 1+1/\ell$ for any positive integer $\ell$. In this paper, we find the remaining part of the full spectrum of FO formulae of quantifier depth $3$. It differs slightly from the ($\alpha>1$)-part of the spectrum for the MSO formulae of depth $3$.
\begin{theorem}
Let $\alpha>0$. The random graph $G(n,n^{-\alpha})$ does not obey the FO zero-one 3-law if and only if $\alpha\in\{1,\frac{7}{6},\frac{6}{5},\frac{5}{4},\frac{4}{3},\frac{3}{2},2\}$.
\label{FO-new}
\end{theorem}

The statement of Theorem~\ref{3-law} (in comparison with Theorem~\ref{FO-new}) includes three extra values of $\alpha$, namely, $\frac{4}{5},\frac{5}{6},\frac{9}{8}$. It is not so unexpected, since the MSO language (even when we consider only the sentences of quantifier depth at most 3) is much more expressive than the FO language: recall that the fragment of the MSO under study is rich enough since, in particular, connectedness is expressed by the MSO sentence~(\ref{connectedness_MSO}) of quantifier depth 3, in contrast to the respective fragment of the FO. However, it is quite surprising that we have a gap in the values of $\alpha$ of the form $1+1/\ell$: the MSO zero-one 3-law holds for $\ell=7$ but fails for $\ell=8$.
Let us briefly describe the intuition behind both effects.

In the cases $\alpha=\frac{4}{5}$ and $\alpha=\frac{5}{6}$, there are certain graphs $F_{4/5}$ and $F_{5/6}$ such that a.a.s. the existence of a copy of $F_{4/5}$ and a copy of $F_{5/6}$ can be expressed by existential MSO sentences with one monadic quantifier and with the FO parts of quantifier depth two,\footnote{Such sentences appear because in FO with quantifier depth $2$, given a set $X$, the following properties can be expressed and are non-trivial for our purposes: 1) the subgraph induced on $X$ has an isolated (or a universal) vertex, 2) there is a vertex outside $X$ which is a common neighbor of all the vertices in $X$, see Section~\ref{proof_simple_MSO} for details.} in contract with the FO language. (In particular, it is well known that, in order to express the property of containing an isomorphic copy of $F$ in FO, one needs the quantifier depth to be at least the number of vertices in $F$~\cite{VerbZhuk}.) The existence of these graphs implies that there is no MSO zero-one 3-law in these cases (the crucial thing here is that the {\it densities} of these graphs, i.e., halves of the average degrees, are $4/5$ and $5/6$, respectively, cf. Theorem~\ref{threshold}).

For $\alpha=\frac{\ell+1}{\ell}>1$, a.a.s. the random graph is a forest (cf. Section~\ref{pre}). Therefore, we need to deal with sentences that expresses existential properties on acyclic graphs. As in the previous case, these are the `subgraph existence' properties for subgraphs that have densities $1/\alpha$ and thus are trees on $\ell+1$ vertices. The graph on Figure~\ref{figt} has exactly $9$ vertices and can be expressed by an existential MSO sentence with 1 monadic quantifier and with the FO part of quantifier depth 2 (see Section~\ref{proof_simple_MSO}), which exploits some symmetry properties of this graph. This cannot be done for any tree on 8 vertices as none of them has similar properties.\\

By Theorem~\ref{FO-new}, the minimal $k$ for which there exists an FO formula of quantifier depth $k$ and with  infinite spectrum is either $4$ or $5$. In~\cite{Zhuk_Logic}, it is proved that for any $k\geq 4$ and any FO formula of quantifier depth $k$ the intersection of its spectrum with $(1,\infty)$ is finite. Therefore, it is natural to ask the following question:

\begin{prb} Does there exist a FO formula $\varphi$ of quantifier depth $4$ such that $\mathcal{S}(\varphi)\cap (0,1)$ is infinite?
\end{prb}
Note that we do not consider the trivial case $\mathrm{q}(\varphi)=1$. If $\mathrm{q}(\varphi)=2$, then $S(\varphi)\subset\{2\}$ (everywhere in the paper, we write $\subset$ for (not necessarily strict) set-inclusion) for every MSO sentence $\varphi$ (this trivially follows from Ehrenfeucht theorem, see Theorem~\ref{ehren} in Section~\ref{pre}). All the above results are summarized in Table~\ref{Main_Results}.

Theorems \ref{inf_spectrum}, \ref{3-law}, and \ref{FO-new} are proved in Sections \ref{proof_simple} and \ref{proof2}. In Section \ref{proof_simple}, we construct formulae that show that the corresponding zero-one laws in the three theorems do not hold for the declared values of $\alpha$. More precisely, Theorem~\ref{inf_spectrum} is proved in Section~\ref{proof1}; in Sections~\ref{proof_simple_FO},~\ref{proof_simple_MSO}, we prove sufficiency in Theorems~\ref{FO-new},~\ref{3-law} respectively. In Section \ref{proof2}, we prove that the zero-one laws from Theorems \ref{3-law}, \ref{FO-new} hold for all the values of $\alpha$, not covered in Section \ref{proof_simple}. In particular, in Section~\ref{sec411}, we finish the proof of Theorem~\ref{FO-new}.

\begin{table}[!ht]
\centering
\begin{tabular}{|c|c|c|}

  \hline

  $k$ & $\left|\bigcup_{\FO\,\varphi:\,\qd(\phi)=k}\mathcal{S}(\varphi)\right|$   &  $\left|\bigcup_{\MSO\,\varphi:\,\qd(\phi)=k}\mathcal{S}(\varphi)\right|$  \\

  \hline

  $2$ & $\{2\}$  & $\{2\}$ \\

  \hline

  $3$ & $\{1,\frac76,\frac65,\frac54,\frac43,\frac32,2\}$  & $\{\frac 45,\frac56,1,\frac98,\frac76,\frac65,\frac54,\frac43,\frac32,2\}$ \\

  \hline

  $4$ & finite or not?  & $\infty$ \\

  \hline

  $5$ & $\infty$  & $\infty$ \\

  \hline

\end{tabular}
\vspace{-0.2cm}

\caption{\small Spectra for different values of quantifier depth.}
\label{Main_Results}
\end{table}

\section{Preliminaries}
\label{pre}

We start with the necessary  notations and auxiliary statements.\\ 

Throughout this paper, we denote the vertex set of a graph $G$ by $V(G)$ and its edge set by $E(G)$ (i.e., $G=(V(G),E(G))$). For the random graph $G(n,p)$, we simply put $G(n,p) := (V_n,E)$.
The degree of $v\in V(G)$ is denoted by $\deg(v)$. \\




We need two concentration results about $G(n,p)$, embodied in the next two lemmas. The first one concerns the degrees of the vertices in $G(n,p)$.
\begin{lemma}
Let $p=n^{-\alpha}$ for some $0<\alpha<1$. Then for some $c>0$, a.a.s. for any $v\in V_n$ we have
$$
\deg(v)\in \left[ n^{1-\alpha}-cn^{1/2-\alpha/2}\sqrt{\ln n},n^{1-\alpha}+cn^{1/2-\alpha/2}\sqrt{\ln n}\right].
$$
\label{degrees}
\end{lemma}
See  the proof in~\cite{Bollobas}, Corollary 3.4.
The next lemma is concerned with the cardinality $\alpha(G(n,p))$ of the largest independent set in $G(n,p)$.
\begin{lemma}\label{ind_set} Let $p = n^{-\alpha}$, $0<\alpha<1$. Then a.a.s.
$$
 \alpha(G(n,p))=2(1-\alpha)n^{\alpha}\ln n-2n^{\alpha}\ln\ln n +O\left(n^{\alpha}\right).
$$
\end{lemma}
For the proof see  \cite{Bollobas}, Theorem 11.28.\\

\textbf{Small subgraphs of the random graph}
\label{pre_graphs}\vskip+0.2cm


Consider a graph $G$ on $v$ vertices and $e$ edges.   The fraction $\rho(G) = \frac{e}{v}$ is called {\it the density} of $G$. Set $\rho^{\max}(G)=\max_{H\subset G}\rho(H)$.  The threshold probability for the property ``$G(n,p)$ contains $G$ as a subgraph'' was determined by B.~Bollob\'{a}s in 1981~\cite{Bol_small_thrshld,Bol_small}. Moreover, the limiting probability of the property at the threshold was estimated in 1989~\cite{BolWier}. Denote by $N_G$ the number of (not necessarily induced) copies of $G$ in $G(n,p)$.


\begin{theorem}
If $p\gg n^{-1/\rho^{\max}(G)}$, then a.a.s. $G(n,p)$ contains a (not necessarily induced) copy of $G$. Moreover, $\frac{N_G}{{\sf E}N_G}\stackrel{{\sf P}}\rightarrow 1$ (for every $\varepsilon>0$, ${\sf P}(|N_G/{\sf E}N_G-1|>\varepsilon)\to 0$ as $n\to\infty$). If $p\ll n^{-1/\rho^{\max}(G)}$, then a.a.s. $G(n,p)$ does not contain any copy of $G$. If $p=c n^{-1/\rho^{\max}(G)}$ for some constant $c>0$, then we have $N_G> 0$ with asymptotic probability $\xi(G)$, where $0<\xi(G)<1$.
\label{threshold}
\end{theorem}

{\it Remark}. Under the assumption that $p=o(1)$, the statement of the theorem holds also for induced copies of $G$ (this can be easily derived from the non-induced case --- see the proof of Lemma~\ref{maximal_extension} below). The proof of this theorem and other important results about the distribution of small subgraphs of the random graph can be found, e.g., in Chapter 3 of \cite{Janson}.

The next lemma follows easily from Theorem~\ref{threshold}.
\begin{lemma}
Let $\alpha>0$. Fix an integer $k$, where $k>1/\alpha$, and let $H$ be a graph on $k$ vertices. Suppose that $\rho^{\max}(H)< 1/\alpha$. Then a.a.s. in $G(n,n^{-\alpha})$ there exists an induced subgraph $H'$ isomorphic to $H$, such that the vertices of $H'$ have no common neighbors outside $H'$.
\label{maximal_extension}
\end{lemma}
Indeed, consider the following set of graphs $\mathcal{H}$ produced from the graph $H$: one graph in $\mathcal{H}$ is obtained from $H$ by adding an extra common neighbor of all the vertices in $H$, and all the others contain $H$ as a spanning subgraph. Let $\tilde H\in\mathcal{H}$. If $\rho^{\max}(\tilde H)>1/\alpha$, then, by Theorem~\ref{threshold}, a.a.s. there are no copies of $\tilde H$ in $G(n,n^{-\alpha})$. If $\rho^{\max}(\tilde H)<1/\alpha$, then there are copies of $\tilde H$ but ${\sf E}N_{\tilde H}\ll{\sf E}N_H$, and then the same a.a.s. holds for $N_H$, $N_{\tilde H}$ by Theorem~\ref{threshold}. Finally, if $\rho^{\max}(\tilde H)=1/\alpha$, then we can consider an $\varepsilon>0$ such that the expected number of copies of $H$ in $G(n,n^{-\alpha})$ is asymptotically bigger than the number of copies of $\tilde H$ in $G(n,n^{-\alpha+\varepsilon})$. In this case, by Theorem~\ref{threshold}, a.a.s. $N_H$ is asymptotically bigger than $N_{\tilde H}$ as well since the property of containing a copy of given subgraph is increasing (see, e.g.,~\cite{Janson}). As $\mathcal{H}$ is finite and its cardinality does not depend on $n$, we get the desired statement.

Indeed, for any collection  $X$ of $k$ vertices, the probability that they have a common neighbor tends to $0$, and this event is independent of the structure of the induced subgraph on $X$. Thus, any given copy of $H$ in $G(n,n^{-\alpha})$ has no common neighbor with probability tending to $1$.



We use the notation $\mathbb N:=\{1,2,\ldots\}$. From Theorem~\ref{threshold}, it follows that if $\ell\in \mathbb N$ and $\alpha\in (1+\frac 1{\ell+1},1+\frac 1{\ell})$, then the following three properties hold:
\begin{itemize}
\item[\textbf{T1}\phantom{$(\ell)$}] The random graph $G(n,n^{-\alpha})$ is a forest a.a.s.
\item[$\mathbf{T2(\ell)}$] A.a.s. any component of $G(n,n^{-\alpha})$ has at most $\ell+1$ vertices.
\item[$\mathbf{T3(\ell)}$] For any integer $K$, a.a.s. for any tree $T$ on at most $\ell+1$ vertices there are at least $K$ components in $G(n,n^{-\alpha})$ which are isomorphic to $T$.
\end{itemize}
If $\ell\in \mathbb N$ and $\alpha=1+\frac 1{\ell}$, then the properties T1, T2$(\ell)$, T3$(\ell-1)$ hold, as well as
\begin{itemize}
\item[$\mathbf{T4(\ell)}$] For any tree $T$ on $\ell+1$ vertices $N_T>0$ with probability $\xi(T)$, $0<\xi(T)<1$.
\end{itemize}

Finally, we need two extension statements. The first lemma is an easy corollary of Spencer's results \cite{Spencer_ext}.
\begin{lemma}
Let $\alpha>0$. Choose a non-negative integer $k$, $k<1/\alpha$, and a positive integer $m$, $m\ge k$. Then a.a.s.  for any vertices $x_1,\ldots,x_m$ in $G(n,n^{-\alpha})$ there is a vertex $z$ which is adjacent to each of $x_1,\ldots,x_k$ and not adjacent to each of $x_{k+1},\ldots,x_m$.
\label{extension}
\end{lemma}

The second lemma is a particular case of Theorem 2 from \cite{Zhuk_neutral}.

\begin{lemma}
Let $\alpha\ge 1/2$. Choose two integers $m$ and $k$, where $k\in \{0,1\}$ and $m\ge 1$. Then a.a.s.  for any vertices $x_1,\ldots,x_m$ in $G(n,n^{-\alpha})$ there is a vertex $z$ which is adjacent to each of $x_1,\ldots,x_k$ and not adjacent to each of $x_{k+1},\ldots,x_m$, that additionally satisfies $\big(N(z)\cap N(x_i)
\big)\setminus\{x_1,\ldots,x_m\}=\varnothing$ for each $i=1,\ldots, m$.
\label{no_neighb_ext}
\end{lemma}
\vskip+0.3cm

\textbf{Ehrenfeucht game}\vskip+0.2cm

An important tool for many of the results on zero-one laws is the Ehrenfeucht game~\cite{AS,Janson,Strange,Survey}. The game $\EHR^{\FO}(A,B,k)$~\cite{Logic2,Libkin,Ehren,Veresh} is played on graphs $A$ and $B$. There are two players, called Spoiler and Duplicator, and a fixed number of rounds $k$. At the $\nu\mbox{-}$th round ($1 \leq \nu \leq k$) Spoiler chooses either a vertex $x_{\nu}$ of $A$ or a vertex $y_{\nu}$ of $B$. Duplicator then chooses a vertex in the other graph.

In the case of monadic second order logic, players can choose subsets as well. Similarly, at the $\nu\mbox{-}$th round ($1 \leq \nu \leq k$) of the game $\EHR^{\MSO}(A,B,k)$~\cite{Logic2,Muller,Zhuk_Logic} Spoiler chooses one of the $A$, $B$. Say, he chooses $A$. Then he either chooses a vertex $x_{\nu}$ or a subset $X_{\nu}$ of $V(A)$. If a vertex is chosen, then Duplicator chooses a vertex of $B$. Otherwise, Duplicator chooses a subset of $V(B)$.

In $\EHR^{\FO}$, at the end of the game the vertices $x_{1},...,x_{k}$ of $A$, $y_{1},...,y_{k}$ of $B$ are chosen. Duplicator wins if and only if the following property holds.

\begin{enumerate}

\item[$\cdot$] For any $i,j\in\{1,\ldots,k\}$, $(x_i\sim x_j)\leftrightarrow (y_i\sim y_j)$, and $(x_i= x_j)\leftrightarrow (y_i=y_j)$.

\end{enumerate}

In $\EHR^{\MSO}$, at the end of the game vertices $x_{i_1},...,x_{i_t}$ of $A$, $y_{i_1},...,y_{i_t}$ of $B$ and subsets $X_{j_1},...,X_{j_{k-t}}$ of $V(A)$, $Y_{j_1},...,Y_{j_{k-t}}$ of $V(B)$ are chosen. Duplicator wins if and only if the following two properties hold.

\begin{enumerate}

\item[$\cdot$] For any $j_1,j_2\in\{i_1,\ldots,i_t\}$, $(x_{j_1}\sim x_{j_2})\leftrightarrow (y_{j_1}\sim y_{j_2})$, and $(x_{j_1}= x_{j_2})\leftrightarrow (y_{j_1}=y_{j_2})$.

\item[$\cdot$] For any $\nu\in\{i_1,\ldots,i_t\}$ and $\mu\in\{j_1,\ldots,j_{k-t}\}$,
$x_{\nu}\in X_{\mu}\leftrightarrow y_{\nu}\in Y_{\mu}$.

\end{enumerate}

The following well-known result establishes the connection between zero-one laws and Ehrenfeucht games (see, e.g.,~\cite{Muller,Janson,Strange,Survey}).
\begin{theorem}
Let $k$ be any positive integer. The random graph $G(n,p)$ obeys the FO zero-one $k$-law if and only if a.a.s. Duplicator has a winning strategy in $\EHR^{\FO}(G(n,p),G(m,p),k)$ as $n,m\to\infty$. The random graph $G(n,p)$ obeys the MSO zero-one $k$-law if and only if a.a.s. Duplicator has a winning strategy in $\EHR^{\MSO}(G(n,p),G(m,p),k)$ as $n,m\to\infty$.
\label{ehren}
\end{theorem}

\section{When the zero-one laws fail}
\label{proof_simple}
This section contains the proofs of the parts of the results which state that the zero-one law does not hold for some value of $\alpha$.
\subsection{Proof of Theorem~\ref{inf_spectrum}}
\label{proof1}

A formula $\varphi$ with $\qd(\varphi)=4$ and infinite spectrum is given below:
\begin{align*}
  \varphi=\exists X\quad \biggl(& \biggl[\exists x_1\exists x_2\forall x & \biggl(X(x_1)\wedge \neg X(x_2)\wedge[(X(x)\wedge x\neq x_1)\leftrightarrow(x\sim x_1\wedge x\sim x_2)]\biggr)\biggr]&\wedge\\
  &\,\,\biggl[\neg\biggl(\exists z\forall x & [X(x)\rightarrow (\exists v \,\, [v\sim z\wedge v\sim x])]\biggr)\biggr]&\wedge\\
  &\,\,\biggl[\exists x_1\exists x_2\exists x_3 & (X(x_1)\wedge X(x_2)\wedge X(x_3)\wedge x_1\sim x_2\wedge x_2\sim x_3\wedge x_1\sim x_3)\biggr]&\biggr).
\end{align*}
Let $m\in\mathbb{N}$, $\alpha=\frac{1}{2}+\frac{1}{2(m+1)}$ and $p=n^{-\alpha}$. Consider the formula
\begin{align*}
  \varphi^{\mathrm{FO}}=\exists x_1\exists x_2\quad\biggl(& x_1\sim x_2\wedge &\\
  &\biggl[\neg\biggl(\exists z\forall x & \biggl[([x\sim x_1\wedge x\sim x_2]\vee x=x_1)\rightarrow (\exists v \,\, [v\sim z\wedge v\sim x])\biggr]\biggr)\biggr]&\wedge\\
  &\biggl[\exists y_1\exists y_2 & (x_1\sim y_1\wedge x_1\sim y_2\wedge x_2\sim y_1\wedge x_2\sim y_2\wedge y_1\sim y_2)\biggr]&\biggr).
\end{align*}

\begin{lemma}
There exists $c\in(0,1)$ such that
$$
 \lim\limits_{n\to\infty}{\sf P}(G(n,p)\models\varphi^{\mathrm{FO}})=c.
$$
\label{from_4_to_5}
\end{lemma}

A proof of Lemma~\ref{from_4_to_5} is given in \cite{Zhuk_inf}  (see the proof of Theorem 1). Obviously, the sets of graphs $\{G\models\varphi\}$, $\{G\models\varphi^{\mathrm{FO}}\}$, for which formulae $\varphi$ and $\varphi^{\mathrm{FO}}$ are true, are the same. Therefore, Lemma~\ref{from_4_to_5} holds for $\varphi$ as well. This implies Theorem~\ref{inf_spectrum}.

\subsection{The FO zero-one 3-law}
\label{proof_simple_FO}

In this section we prove sufficiency in Theorem~\ref{FO-new}, by showing that $G(n,n^{-\alpha})$ does not obey the FO zero-one 3-law if $\alpha\in\{1,7/6,6/5,5/4,4/3,3/2,2\}$. If $\alpha=1$, then $G(n,n^{-\alpha})$ does not obey the FO zero-one 3-law, because $G(n,n^{-\alpha})$ does not contain a triangle with asymptotic probability $c_1\in(0,1)$. Let us prove that $G(n,n^{-\alpha})$ does not obey the FO zero-one 3-law if $\alpha\in\{7/6,6/5,5/4,4/3,3/2,2\}$. Below, for each $\alpha$ from this set, we give a FO formula with quantifier depth at most $3$ each of which has asymptotic probability in $(0,1)$, as can be proved from  $\mathbf{T1}$ -- $\mathbf{T4}(\ell)$. More precisely, for forests with components of size at most $\ell+1$, the sentence enumerated $\ell$) says that a graph contains a given tree of size exactly $\ell+1$. From T1,T2$(\ell)$,T4$(\ell)$, it follows that $G(n,n^{-1-1/\ell})$ contains such a tree with asymptotical probability in $(0,1)$.

We define
$$
P_2(x,y)=([\exists z\quad (x\sim z\wedge y\sim z)]\wedge x\neq y),\quad S(x)=(\exists y \exists z\quad [x \sim y \wedge x \sim z \wedge y \ne z]).
$$

\begin{enumerate}

\item[1)] $\alpha=2$: $\exists x_1\exists x_2\quad (x_1\sim x_2)$.

\item[2)] $\alpha = \frac{3}{2}$:  $\exists x_1\exists x_2\quad P_2(x_1,x_2)$.

\item[3)] $\alpha = \frac{4}{3}$:  $\exists x_1\exists x_2\quad (P_2(x_1,x_2)\wedge S(x_1)).$

\item[4)] $\alpha = \frac{5}{4}$:
$\exists x_1\,\varphi(x_1)$, where
$$
\varphi(x_1)=(S(x_1)\wedge[\forall x_2 \exists x_3\quad x_2 \sim x_1 \rightarrow (x_2 \sim x_3 \wedge x_3 \ne x_1)]).
$$

\item[5)] $\alpha = \frac{6}{5}$:
$\exists x_1\quad \biggl(\varphi(x_1) \wedge\biggl[\exists x_2\quad (P_2(x_1,x_2) \wedge [\exists x_3\quad (x_3 \not\sim x_1 \wedge x_3 \sim x_2)])\biggr]\biggr)$.

\item[6)] $\alpha = \frac{7}{6}$: $\exists x_1\quad \biggl(\varphi(x_1) \wedge \biggl[\forall x_2\quad (P_2(x_1,x_2) \rightarrow [\exists x_3\quad (x_3 \nsim x_1 \wedge x_3 \sim x_2)])\biggr]\biggr)$.

\end{enumerate}
In $G(n,n^{-\alpha})$ with $\alpha>1$, each of these formulas are fulfilled when the graph contains a copy of the tree depicted on Figure~\ref{fig0}. Moreover, each of these trees are minimal among all trees that satisfy the corresponding property.
\begin{figure}
\begin{center}
\begin{tikzpicture}[scale=0.7]


\draw[thick] (1,-2) -- (1,-3.5);

\draw[thick] (2,-3.5) -- (2.5,-2) --(3,-3.5);
\draw[thick] (3.5,-2)--(4,-3.5) -- (5,-3.5) -- (5.5,-2);
 (2,-2);
 \node[fill=black, circle, inner sep = -0pt, minimum size=4pt] at (1,-2) {.};
 \node[fill=black, circle, inner sep = -0pt, minimum size=4pt] at (1,-3.5) {.};
 \node[fill=black, circle, inner sep = -0pt, minimum size=4pt] at (2,-3.5) {.};
 \node[fill=black, circle, inner sep = -0pt, minimum size=4pt] at (3,-3.5) {.};
 \node[fill=black, circle, inner sep = -0pt, minimum size=4pt] at (2.5,-2) {.};

 \node[fill=black, circle, inner sep = -0pt, minimum size=4pt] at (3.5,-2) {.};
 \node[fill=black, circle, inner sep = -0pt, minimum size=4pt] at (4,-3.5) {.};
 \node[fill=black, circle, inner sep = -0pt, minimum size=4pt] at (5,-3.5) {.};
 \node[fill=black, circle, inner sep = -0pt, minimum size=4pt] at (5.5,-2) {.};
\end{tikzpicture}
\ \ \ \begin{tikzpicture}[scale=0.7]


\draw[thick] (2,-1) -- (1,-0.6)-- (0,0) -- (1,0.6) --(2,1);
 \node[fill=black, circle, inner sep = -0pt, minimum size=4pt] at (2,-1) {.};
 \node[fill=black, circle, inner sep = -0pt, minimum size=4pt] at (1,-0.6) {.};
 \node[fill=black, circle, inner sep = -0pt, minimum size=4pt] at (1,0.6) {.};
 \node[fill=blue, circle, inner sep = -0pt, minimum size=5.5pt] at (0,0) {.};
 \node[fill=black, circle, inner sep = -0pt, minimum size=4pt] at (2,1) {.};
\end{tikzpicture}
\ \ \ \begin{tikzpicture}[scale=0.7]


\draw[thick] (2,-1) -- (1,-0.6)-- (0,0) -- (1,0.6) --(2,1)--(3,0.5);
 \node[fill=black, circle, inner sep = -0pt, minimum size=4pt] at (2,-1) {.};
 \node[fill=black, circle, inner sep = -0pt, minimum size=4pt] at (1,-0.6) {.};
 \node[fill=black, circle, inner sep = -0pt, minimum size=4pt] at (1,0.6) {.};
 \node[fill=blue, circle, inner sep = -0pt, minimum size=5.5pt] at (0,0) {.};
 \node[fill=black, circle, inner sep = -0pt, minimum size=4pt] at (2,1) {.};
 \node[fill=black, circle, inner sep = -0pt, minimum size=4pt] at (3,0.5) {.};
\end{tikzpicture}
\ \ \ \begin{tikzpicture}[scale=0.7]


\draw[thick] (3,-0.5)--(2,-1) -- (1,-0.6)-- (0,0) -- (1,0.6) --(2,1)--(3,0.5);
 \node[fill=black, circle, inner sep = -0pt, minimum size=4pt] at (2,-1) {.};
 \node[fill=black, circle, inner sep = -0pt, minimum size=4pt] at (1,-0.6) {.};
 \node[fill=black, circle, inner sep = -0pt, minimum size=4pt] at (1,0.6) {.};
 \node[fill=blue, circle, inner sep = -0pt, minimum size=5.5pt] at (0,0) {.};
 \node[fill=black, circle, inner sep = -0pt, minimum size=4pt] at (2,1) {.};
 \node[fill=black, circle, inner sep = -0pt, minimum size=4pt] at (3,0.5) {.};
 \node[fill=black, circle, inner sep = -0pt, minimum size=4pt] at (3,-0.5) {.};
\end{tikzpicture}

\end{center}

\caption{Trees corresponding to the formulas 1)-- 6). The big blue circles denote  $x_1$ in 4)--6).}\label{fig0}
\end{figure}
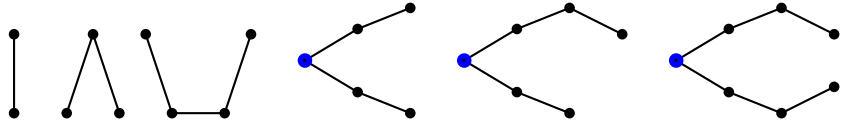

\subsection{The MSO zero-one 3-law}
\label{proof_simple_MSO}

As we have seen in the previous section, if $\alpha\in\{7/6,6/5,5/4,4/3,3/2,2\}$, then $G(n,n^{-\alpha})$ does not obey the FO zero-one 3-law.\\

Let $\alpha=9/8$ and $p=n^{-\alpha}$. Consider the formula
\begin{align*}
  \varphi=\exists X& \biggl(\biggl[\exists x\quad \biggl(X(x)\wedge[\forall y\quad ([X(y)\wedge x\neq y]\rightarrow x\sim y)]\biggr)\biggr]&\wedge\\
  &\,\,\,\,\biggl[\forall x\quad \biggl(X(x)\rightarrow[\exists y\quad (\neg X(y)\wedge x\sim y )]\biggr)\biggr]&\wedge\\
  &\,\,\,\,\biggr[\forall y\quad \biggl([\neg X(y)\wedge(\exists x \quad [X(x)\wedge x\sim y])]\rightarrow
  [\exists z\quad (\neg X(z)\wedge y\sim z)]\biggr)\biggr]&\biggr).
\end{align*}
The formula $\varphi$ says that there exists a set $X$ with a {\it universal vertex} (universal vertex is adjacent to all the other vertices in this set) such that every $x\in X$ has a neighbor outside $X$ and every vertex in the {\it neighborhood} of $X$ (the set of all neighbors of vertices in $X$) has a neighbor outside $X$. Obviously, if a graph $G$ is a forest and $\varphi$ is true for $G$, then $G$ contains a component $T$ isomorphic to the graph depicted on Figure~\ref{figt}.
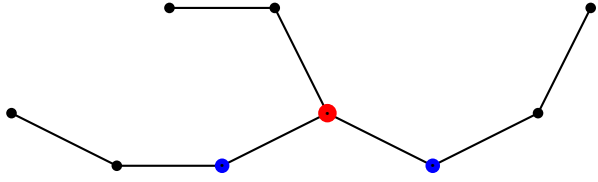
\begin{figure}
\begin{center}
\begin{tikzpicture}[scale=0.7]


\draw[thick] (-4,-2) -- (-2,-3)-- (0,-3) -- (2,-2) --(1,0)--(-1,0);
\draw[thick] (2,-2)--(4,-3) -- (6,-2) -- (7,0);
 (2,-2);
 \node[fill=black, circle, inner sep = -0pt, minimum size=4pt] at (2,-2) {.};
 \node[fill=black, circle, inner sep = -0pt, minimum size=4pt] at (-4,-2) {.};
 \node[fill=black, circle, inner sep = -0pt, minimum size=4pt] at (-2,-3) {.};
 \node[fill=blue, circle, inner sep = -0pt, minimum size=5.5pt] at (0,-3) {.};
 \node[fill=red, circle, inner sep = -0pt, minimum size=7pt] at (2,-2) {.};
 \node[fill=black, circle, inner sep = -0pt, minimum size=4pt] at (1,-0) {.};
 \node[fill=black, circle, inner sep = -0pt, minimum size=4pt] at (-1,0) {.};
 \node[fill=blue, circle, inner sep = -0pt, minimum size=5.5pt] at (4,-3) {.};
 \node[fill=black, circle, inner sep = -0pt, minimum size=4pt] at (6,-2) {.};
 \node[fill=black, circle, inner sep = -0pt, minimum size=4pt] at (7,0) {.};
\end{tikzpicture}
\end{center}

\caption{Tree $T$ corresponding to the formula $\varphi$. Big red circle corresponds to the universal vertex from $X$, medium sized blue circles correspond to two other vertices in $X$.}\label{figt}
\end{figure}

Moreover, if there is a component isomorphic to $T$ in $G$, then $G\models\varphi$. As $G(n,p)$ has the properties \textbf{T1}, $\mathbf{T2(8)}$, $\mathbf{T3(7)}$ and $\mathbf{T4(8)}$ (see Section~\ref{pre_graphs}), a.a.s. $G(n,p)\models\varphi$ if and only if in $G(n,p)$ there is a component isomorphic to $T$. By the property $\mathbf{T4(8)}$, ${\sf P}(G(n,p)\models\varphi)\to c_2\in (0,1)$ as $n\to\infty$. Therefore, $G(n,p)$ does not obey the MSO zero-one 3-law.\\

Let $\alpha=5/6$ and $p=n^{-\alpha}$. Consider the formula
\begin{align*}
  \varphi=\exists X& \biggl(\biggl[\forall x\quad \biggl(X(x)\rightarrow[(\exists y\quad [X(y)\wedge x\sim y])\wedge(\exists y\quad [X(y)\wedge x\neq y \wedge x\nsim y])]\biggr)\biggr]&\wedge\\
  &\,\,\,\,\biggl[\exists z\quad \biggl(\neg X(z)\wedge [\forall x\quad (X(x)\rightarrow z\sim x)]\biggr)\biggr]&\biggr).
\end{align*}
Let $H$ be a graph on $5$ vertices and $6$ edges which is the union of two triangles sharing one vertex. Then $G\models\varphi$ if and only if $G\supseteq H$ (not necessarily induced). By Theorem~\ref{threshold}, $G(n,p)\supset H$ with asymptotic probability $c_3\in (0,1)$. Therefore, ${\sf P}(G(n,p)\models\varphi)\to c_3$ as $n\to\infty$ and so $G(n,p)$ does not obey the MSO zero-one 3-law.\\

Finally, let $\alpha=4/5$ and $p=n^{-\alpha}$. Consider the formula
\begin{align*}
  \varphi=\exists X&\biggl(\biggl[\exists x\quad (X(x)\wedge[\forall y\quad ([X(y)\wedge x\neq y]\rightarrow x\sim y)])\biggr]&\wedge\\
  &\,\,\,\,\biggl[\exists x\exists y\quad (X(x)\wedge X(y)\wedge  x\neq y \wedge x\nsim y)\biggr]&\wedge\\
  &\,\,\,\,\biggl[\exists z\quad (\neg X(z)\wedge [\forall x\quad (X(x)\rightarrow z\sim x)])\biggr]&\biggr).
\end{align*}
Let $H$ be a {\it diamond graph} (two triangles sharing an edge). Obviously, $G\models\varphi$ if and only if $G\supset H$. By Theorem~\ref{threshold}, ${\sf P}(G(n,p)\models\varphi)\to c_4\in(0,1)$ as $n\to\infty$ and so $G(n,p)$ does not obey the MSO zero-one 3-law.\\

\section{When the zero-one laws hold}
\label{proof2}

In this section, we describe the Duplicator's winning strategy for the remaining values of $\alpha$ from Theorems~\ref{3-law} and \ref{FO-new}. The strategy for proving Theorem~\ref{3-law} is based on a classification of subset-pairs and subset-vertex-pairs, which is carried out in Sections~\ref{alpha<1} and~\ref{alpha>1}.

Let $G$ be a graph, and $X$ be a subset of $V(G)$. We denote by $\bar{X}$ and $\bar{G}$ the set $V(G)\setminus X$ and the complement of $G$ respectively. For a given vertex $v\in V(G)$, its set of neighbors is denoted by $N(v)$. We denote by $\deg_X(v)$ the number of vertices from $X$ that are adjacent to $v$. If $\deg(v)=0$, then we call $v$ \textit{isolated}.

\subsection{Proof of Theorem~\ref{FO-new}}\label{sec411}

Let us recall that $G(n,n^{-\alpha})$ obeys the FO zero-one 3-law if $\alpha<1$~\cite{Zhuk_law}. So, due to Theorem~\ref{FO01}, for this section, we assume that $\alpha= 1+\frac 1{\ell}$, where $\ell\ge 7$; to prove necessity of the condition in Theorem 5 we have to now prove that $G(n,n^{-\alpha})$ obeys the zero-one 3-law.
Let $T$ be a component in $G$ containing $x_1$, the choice of Spoiler in the first round. If in $H$ there is a component $F$ isomorphic to $T$, then Duplicator chooses $y_1=\varphi(x_1)$, where $\varphi:T\rightarrow F$ is an isomorphism. In what follows, we discuss the choice of Duplicator in case there is no such component. Note that by T3(6) a.a.s. for any tree on at most $7$ vertices there is a component in $H$ isomorphic to it.

The {\it distance} between two vertices $u$ and $v$ in a connected graph is the minimum edge length of a path connecting $u$ and $v$. It is denoted by $d(u,v)$.  If there is no path between $u$ and $v$, we write $d(u,v)=\infty$ (for any real $d$, we let $d<\infty$). For a tree $R$ and a vertex $v\in R$ define the \textit{leaf distance set} $L(v)\subset \mathbb N\cup\{0\}$ as the set of all distances from $v$ to the leaves of $R$. Note that $0\in L(v)$ iff $v$ is a leaf itself. We call a subset of $\mathbb N\cup\{0\}$ \textit{admissible}, if it is a leaf distance set for some tree and its vertex. It is easy to see that a set is not admissible iff it contains both $0$ and $1$, and has cardinality at least 3.

It is not difficult to check that for each admissible $L$, there is a tree $R(L)$ on at most $7$ vertices and a vertex $v\in R(L)$ in that tree such that, first, $L\cap \{0,1,2\} = L(v)\cap\{0,1,2\}$, and, second, $3\in L(v)$ iff $L\setminus \{0,1,2\}$ is nonempty. Indeed, in the worst case, $L\cap\{0,1,2\}=\{1,2\}$ and $L\setminus\{0,1,2\}\neq\varnothing$. But this case requires exactly 7 vertices (for such $L$, $R(L)$ is a union of $P_2,P_3,P_4$, where $P_s$ is a simple path on $s$ vertices, sharing a common first vertex $v$). Duplicator finds a (tree) component $K$ in $H$ isomorphic to $R(L(x_1))$ and the vertex $y_1\in K$ that corresponds to $x_1$.

If in the second round Spoiler chooses a vertex (say, $x_2\in V(G)$), then Duplicator chooses a vertex $y_2\in V(H)$ such that

\begin{itemize}
\item If $d(x_1,x_2)\le 2$, then $d(y_1,y_2)=d(x_1,x_2)$ and $y_2$ is a leaf if and only if $x_2$ is a leaf;
\item if $d(x_1,x_2)>2$ then $d(y_1,y_2)=\infty$;
\item $y_2$ is isolated if and only if $x_2$ is isolated.
\end{itemize}

By the choice of $K$ and $y_1$, it is clear that such $y_2$ exists.
It is not difficult to see that in the third round Duplicator has a winning strategy as well. Together with the proof of sufficiency given in Section~\ref{proof_simple_FO}, Theorem~\ref{FO-new} is proved.

\subsection{Classifications of sets and strategies for Duplicator}\label{sec41}

In this section, we prove Theorem~\ref{3-law}.

We first treat the degenerate set-choices $|X|\in\{0,1,n-1,n\}$. Interchanging $X,Y$ and $\bar X,\bar Y$, we may w.l.o.g. assume that  $|X|\in \{0,1\}$. If $X=\emptyset$, then Duplicator chooses $Y=\emptyset$, and in the following round plays as if the previous round was not played. If $X=\{v\}$, then Duplicator chooses $Y=\{w\}$, where $w$ is taken according to the strategy of Duplicator in case Spoiler chose a vertex $v$ (and not a set $\{v\}$). In the next round Duplicator plays as if in the previous round the players chose vertices $v,w$. Clearly, if Duplicator has a winning strategy for the case when Spoiler chooses either nontrivial sets or vertices, then Duplicator has a winning strategy for the degenerate choices. In what follows we therefore assume that Spoiler chooses only {\it nondegenerate sets} $X$, that is, such that both $|X|\geq 2$ and $|\bar X|\geq 2$.

Note also that in case $\alpha>1$, in view of Theorem~\ref{MSO01}, we may assume that $\alpha = 1+\frac 1{\ell}$ for $\ell\in \mathbb N$, which we do tacitly for the rest of the section.

\subsubsection{Pairs of complementary subsets of $V_n$}
\label{classification}

In the next paragraph, we classify the vertices of $V_n$ into six types with respect to a subset $X$ and give  names to the classes.
The intuition behind this classification is the following. Assume, that you consider FO sentences having quantifier depth $2$ (as if you skip one monadic quantifier). Then, these sentences divide the set of all graphs into the following classes of elementary equivalence (see, e.g.,~\cite{Survey}, Page 39): cliques, empty graphs, graphs with isolated vertices having at least one edge, graphs with universal vertices having at least one non-edge, and all the other (common) graphs. Recovering one monadic variable $X$, we consider now the existential fragment of MSO with only one monadic quantifier. Distinguishing equivalence classes w.r.t. these sentences requires, in particular, considering the above classes of graphs induced by $X$. This leads to the following crucial definitions.

We say that $v\in X$ is {\it $X$-inside-dominating (respectively $X$-outside-dominating)}  if $v$ is adjacent to all the vertices in $X\setminus\{v\}$ (respectively $\bar X$). We say that $v\in X$ is {\it $X$-inside-isolated (respectively $X$-outside-isolated)}  if $v$ is nonadjacent to all the vertices in $X$ ($\bar X$). Otherwise we say that $v\in X$ is {\it $X$-inside-common (respectively $X$-outside-common)}.

In some cases it is more convenient for us not to specify whether $v\in X$ or $v\in \bar X$. If $v\in V$ is adjacent to all vertices of $X$, except maybe itself, then we say that $v$ is {\it $X$-dominating}. If $v\in V$ is nonadjacent to all vertices from $X$, then we say that $v$ is {\it $X$-isolated}. Otherwise, we say that $v$ is \textit{$X$-common}.

Based on this classification of vertices, we define the \textit{type} of a vertex $v\in X$. The types are all possible pairs of properties, where the first property in the pair is one of this: ``inside-dominating'', ``inside-common'', ``inside-isolated'', and the second one is one of the analogous outside-properties. The {\it type} of $v\in X$ is the pair of properties that $x$ satisfy w.r.t $X$.
We stress that the type is defined with respect to a subset. It will be clear from the context w.r.t. which subset the type is defined.\\




%

The crucial part of the proof of Theorem \ref{3-law} is the classification of pairs $X,\bar X$ based on the types their vertices have. The \textit{type of a pair} $X, \bar X$ is specified by the following two parameters: \vskip+0.2cm\noindent
1. The list of \textit{all} the types that the vertices of $X$ have.\\
2. The list of \textit{all} the types that the vertices of $\bar X$ have.\\

We illustrate the importance of this classification in the next section.

\subsubsection{Duplicator's strategy if Spoiler chooses a set in the first round}\label{sec413}

In this section and Section~\ref{sec414} we finish the proof of Theorem~\ref{3-law} modulo some classification results, proved in Section~\ref{alpha<1} and~\ref{alpha>1}.
Let $G,H$ be two graphs.

\begin{lemma}
Suppose that in the first round of $\EHR^{\MSO}(G,H,3)$ two (nontrivial) subsets $X\subset V(G)$, $Y\subset V(H)$ are chosen. If
the types of the pairs $X,\bar X$ and $Y, \bar Y$ are the same,
then in the last two rounds Duplicator has a winning strategy.
\label{main_strategy}
\end{lemma}

Indeed, if in the second round Spoiler chooses another set, say, $X'\in V(G)$, then Duplicator chooses a subset $Y'\in V(H)$ such that $Y\cap Y'$ ($Y\cap \bar Y'$, $\bar Y\cap Y'$, $\bar Y\cap \bar Y'$) is nonempty iff $X\cap X'$ ($X\cap \bar X'$, $\bar X\cap X'$, $\bar X\cap \bar X'$) is nonempty. Then Duplicator obviously has a winning strategy in the third round.

If in the second round Spoiler chooses a vertex $v$, say, in $X$, then Duplicator chooses a vertex  $w\in Y$ of the same type as $v$. Again, in the third round Duplicator obviously has a winning strategy.\\

Therefore, to be able to apply Lemma \ref{main_strategy}, we have to prove that for each admissible $\alpha$ each type of a pair appears with asymptotic probability either 0 or 1 in $G(n,n^{-\alpha})$. For $\alpha<1$ this is done in Section~\ref{alpha<1}. When $\alpha>1$, the Duplicator's strategy for choosing the pair of type chosen by Spoiler is described in Section~\ref{alpha>1}.

\subsubsection{Duplicator's strategy if Spoiler chooses a vertex in the first round}\label{sec414}


Let $X \subset V(G)$, $x\in V(G)$, $Y\subset V(H)$, $y\in V(H)$. We say that the pair $(x,X)$ \textit{is equivalent to} the pair $(y,Y)$, if

\begin{itemize}

\item $x\in X$ iff $y\in Y$;

\item the type of $x$ w.r.t. $X$ is the same as the type of $y$ w.r.t. $Y$ (see Section \ref{classification}).
\end{itemize}





The importance of this definition is justified by the following easy lemma.
\begin{lemma}
Suppose that in the first two rounds of $\EHR^{\MSO}(G,H,3)$ a vertex $x$ and a set of vertices $X$ were chosen in $G$, and a vertex $y$ and a set of vertices $Y$ were chosen in $H$, no matter who chose what, and when. If $(x,X)$ is equivalent to $(y,Y)$, then in the last round Duplicator has a winning strategy.
\label{D_strat_vertex}
\end{lemma}
We leave the proof of the lemma to the reader. Assume that Spoiler chooses a subset $X\subset V(G)$ in the second round, and the vertices $x\in V(G),y\in V(H)$ were chosen in the first round.  Because of Lemma~\ref{D_strat_vertex}, it is enough to show that for all admissible values of $\alpha$ in Theorems~\ref{3-law} and \ref{FO-new} it is a.a.s. possible for Duplicator to choose the subset $Y\subset V(H)$, such that the pair $(y,Y)$ is equivalent to $(x,X)$.


  For each $\alpha\ne 1+\frac 1{\ell}$, where $\ell=1,\ldots, 6$, the graph $G(n,p)$ obeys the FO zero-one 3-law~(Theorem~\ref{FO-new} and \cite{Zhuk_law}). Therefore, by Theorem~\ref{ehren}, a.a.s.  Duplicator has a winning strategy in $\EHR^{\FO}(G(n,p(n)),G(m,p(n)),3)$ as $n,m\to\infty$. Let $y_1$ be a vertex chosen by Duplicator according to its (a.a.s.) winning strategy in the first round. If in the second round Spoiler chooses a vertex (say, $x_2\in V_n)$, then Duplicator chooses a vertex $y_2$ according to its (a.a.s.) winning strategy. Obviously, a.a.s. in the third round Duplicator wins. Assume that in the second round Spoiler chose a set (say, $X_2\subset V_n)$. We a.a.s. have $\mathrm{deg}(x_1),\mathrm{deg}(y_1)\in[0,n-4]$. Moreover, any winning strategy  for the FO zero-one 3-law must satisfy  $\deg(x_1)\ge i$ iff $\deg(y_1)\ge i$ for $i=1,2$. Using these properties, choose $Y_2\subset V_m$ in the following way: $Y_2$ contains $y_1$ iff $X_2$ contains $x_1$; $Y_2$ ($\bar Y_2$) contains one neighbor of $y_1$ iff $X_2$ ($\bar X_2$) contains at least one neighbor of $x_1$, and the same for non-neighbors.
  Then Duplicator chooses $Y_2$. By Lemma~\ref{D_strat_vertex}, a.a.s. Duplicator has a winning strategy in the last round.\\

  Concluding Section \ref{sec41}, we remark that for the proof of Theorem~\ref{3-law} it is sufficient to show that for each $\alpha>0,$ where $ \alpha\ne 4/5, 5/6$ and $1+\frac 1{\ell}$ for  $l\in \{1,\ldots, 6\}\cup \{8\}$, each type of a pair $X,\bar X$ appears with asymptotic probability either 0 or 1 in $G(n,n^{-\alpha})$. Then the application of Lemma \ref{main_strategy} will finish the proof.

\subsection{Pairs of subsets for $\alpha<1$}\label{alpha<1}
For a graph $G$ and a subset $X$ of its vertices, we denote by $G|_X$ the induced subgraph of $G$ on $X$. We say that $X$ is {\it complete}, if $G|_X$ is a complete graph, and {\it independent}, if $G|_X$ is an empty graph. A subset $X\subset V$ is {\it dense} ({\it sparse}) if it has an $X$-inside-dominating vertex ($X$-inside-isolated vertex), but is not complete (empty). Finally, we say that $X$ is {\it common} if neither it has an inside-dominating nor an inside-isolated vertex. Remark that $X$ cannot contain an inside-dominating and inside-isolated vertex at the same time.\\

Note that the set of \textit{inside} types of vertices in $X$ is determined for $X$ in each of the five classes above. Therefore, to determine the type of a pair $X, \bar X$, it is sufficient to determine to which of the five classes each of $X$ and $\bar X$ belong, and the set of \textit{outside} types of vertices from $X$ and from $\bar X$.

In this section, we determine \textit{all} the types of pairs that appear a.a.s. in $G(n,n^{-\alpha})$ for $\alpha<1$,  $\alpha\neq 4/5, 5/6$ (which we assume for the rest of Section \ref{alpha<1}), together with the ranges of $\alpha$ for which they do appear. We prove that the classification is \textit{complete}, i.e., that all the other types do not appear a.a.s. As we have discussed in Section \ref{sec413}, this classification is essential for the strategy of Duplicator in the Ehrenfeucht game.\\

\subsubsection{Auxiliary lemmas}

Let $p=n^{-\alpha}$. The next several lemmas will help us to limit the number of possible types.\\

\begin{lemma}\label{sizes}
Let $\alpha \in (0,1)$ and put $p=n^{-\alpha}$. Take a subset $X\subset V_n$ in $G(n,p)$. The following bounds on the size of $X$ hold a.a.s.: \\
1. If $X$ is complete, then $|X|\le \lfloor 2/\alpha \rfloor +1$.\\
2. If $X$ is dense, then $|X|=O(n^{1-\alpha})$.\\
3. If $X$ is sparse, then $|X|= n-\Omega(n^{1-\alpha})$.\\
4. If $X$ is independent, then $|X| = O(n^{\alpha}\log n)$.
\end{lemma}

\begin{proof} Statement 1 follows from Theorem \ref{threshold}. Both statements 2 and 3 follow from Lemma \ref{degrees}. Statement 4 follows from Lemma \ref{ind_set}.
\end{proof}

Lemma \ref{sizes} allows us to deduce that some pairs of types cannot occur as the types of a pair of the form  $X, \bar X$. E.g., By 1. in Lemma~\ref{sizes}, $X,\bar X$ cannot be both dense since then the sum of their cardinalities is constant. The pair ``dense--sparse'' is not excluded by Lemma \ref{sizes}, but is impossible in most situations, as shown in the next lemma.

\begin{lemma}
Let $\alpha \in (0,1)$ and put $p=n^{-\alpha}$.

1. A.a.s. there are no two distinct vertices $u, v$ in $G(n,p)$ with $N(u)\subset N(v)$.

2. A.a.s. in $G(n,p)$ there are no two distinct vertices $u,  v$ such that $v$ is an $X$-dominating vertex and $u$ is an $\bar X$-dominating or $\bar X$-isolated vertex for some $X\subset V_n$.

3. If $\alpha<1/2$, then a.a.s. in $G(n,p)$ there are no two distinct vertices $u,v$, such that $v$ is $X$-isolated and $u$ is $\bar X$-isolated for some $X\subset V_n$.

\label{no_dense}
\end{lemma}

\begin{proof} 1. The probability that for some distinct $u,v\in V$ we have $N(u)\setminus \{v\}\subset N(v)\setminus\{u\}$ is at most $$n^2(1-p(1-p))^{n-2}\le n^2e^{-(n-2)p(1-p)} = o(1)\ \text{  as }\ n\to \infty.$$

2. By Lemma \ref{degrees}, a.a.s. $\deg v+\deg u = O(n^{1-\alpha})$ for all pairs $u,v\in V_n$. Therefore, a.a.s. there are no pairs $v,u\in V$ where $v$ is $X$-dominating and $u$ is  $\bar X$-dominating. The second possibility   is ruled out by the proof of part 1 of this lemma.


3. The probability that there exist $u,v\in V_n$ such that $u$ is $X$-isolated and $v$ is $\bar X$-isolated, is at most  $n^2(1-p^2)^{n-2}\le n^2e^{-p^2(n-2)}$. This expression tends to zero if $\alpha<1/2$ and $n\to \infty$.
\end{proof}
If $X$ is dense then there is an $X$-dominating vertex. Thus, there is no {\it other} $\bar X$-isolated vertex, and therefore $\bar X$ cannot be sparse, unless there is a vertex $v$ which is both $X$-dominating and  $\bar X$-isolated. If $v$ is $X$-dominating and $\bar X$-isolated, or vice versa, then we say that $v$ is \textit{$X$-special}. Note that due to the concentrations of degrees (Lemma \ref{degrees}) a.a.s. no vertex is  both $X$- and $\bar X$-dominating or both $X$- and $\bar X$-isolated.

\subsubsection{Classification of subsets without special vertices}

For a moment we consider only the sets $X$ for which there are no special vertices $v$. The somewhat special case of $X$-special vertices we treat later.

The graph on Figure~\ref{fig2} represents the possibilities that are not ruled out by Lemma \ref{sizes} and 2. of Lemma \ref{no_dense}. The vertices are the five possible types of subsets, and the edges are represented by dashed lines. A non-edge between two vertices means that there cannot be a pair $X,\bar X$ where $X$ and $\bar X$ belong to the types of sets represented by the corresponding vertices (under the assumption that there are no $X$-special vertices). Note that there are two loops in the graph.

The list below the name of the type give the possible types of vertices from the set. All pairs of types of sets and/or vertices that do not appear in Figure~\ref{fig2}
do not appear in $G(n,p)$ a.a.s. (assuming that the set does not have special vertices). In particular, complete and independent sets can have only outside-common vertices, since any non-outside-common vertex in a complete or independent set is special.

\begin{figure}
\begin{center}  \includegraphics[width=120mm]{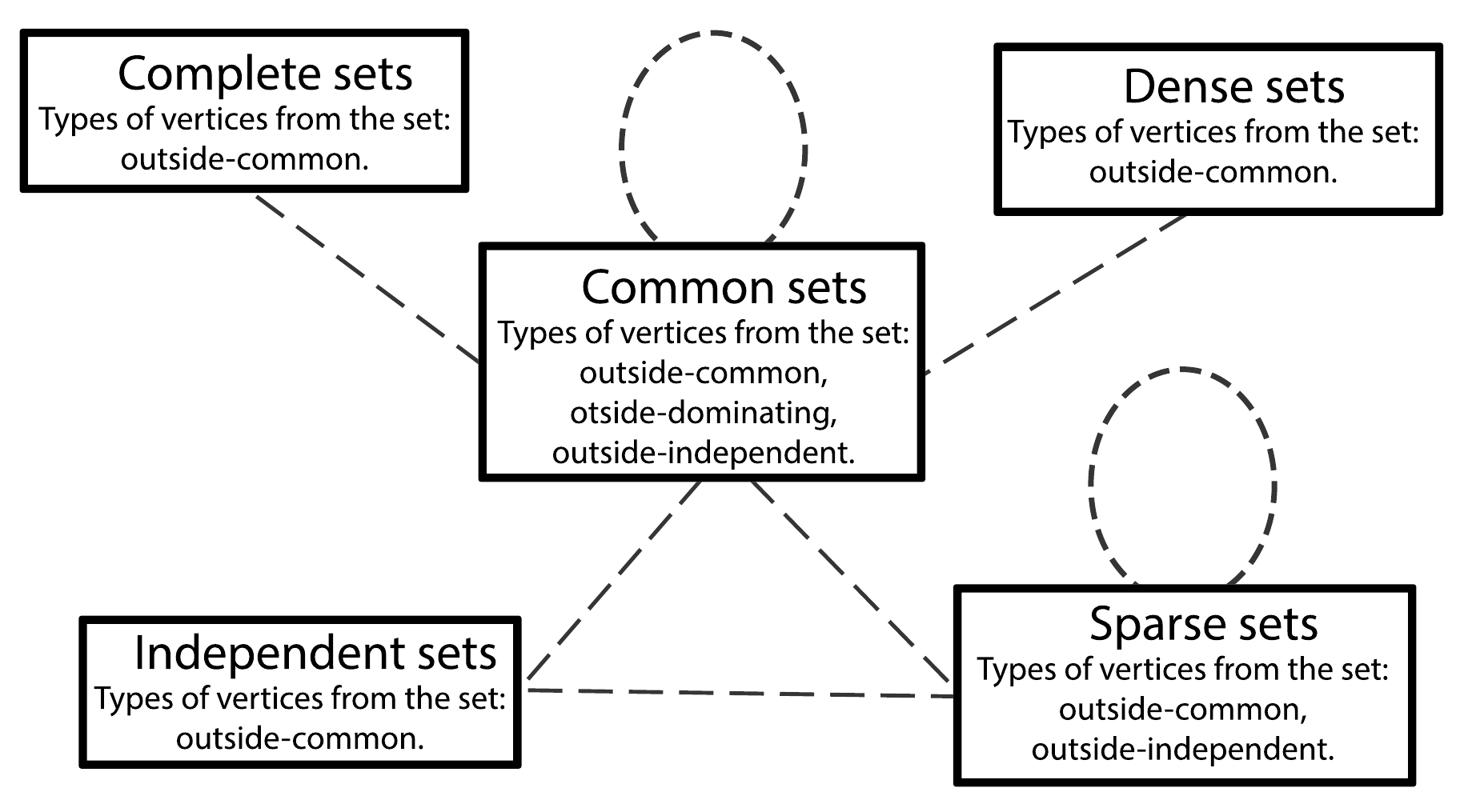}\label{pic1}  \caption{The pairs of $X,\bar X$ and the types of their vertices that are allowed by Lemmas~\ref{sizes} and~\ref{no_dense}.}\label{fig2} \end{center}
\end{figure}
Unfortunately, Figure~\ref{fig2} does not contain all the information that we need: it does not specify which types of vertices can appear \textit{simultaneously} in $X$ or in $\bar X$, provided that $X,\bar X$ are of a given type. The following two lemmas allow us to refine the classification.
Recall that a subset $X$ \textit{nontrivial}, if both $X$ and $\bar X$ have cardinalities at least $2$.
\begin{lemma}
Let $\alpha \in (0,1)$ and put $p=n^{-\alpha}$.
A.a.s. any nontrivial subset $X\subset V_n$ has an $X$-outside-common vertex.
\label{common_vertices}
\end{lemma}

\begin{proof} Let $X\subset V$ be a nontrivial set of cardinality $i$. The probability that one of the vertices of $X$ is outside-common is $1-(p^{n-i}+(1-p)^{n-i})^i$. Therefore, the probability that there exists a nontrivial $X$ with no $X$-common vertices is at most
$$
\sum_{i=2}^{n-2} {n\choose i} (p^{n-i}+(1-p)^{n-i})^i\leq \sum_{i=2}^{n-2} {n\choose i} ((1-p)^{n-i-1})^i\leq 2\sum_{i=2}^{\lfloor n/2\rfloor} {n\choose i} ((1-p)^{n-i-1})^i\leq
$$
$$
2\sum_{i=2}^{\lfloor n/2\rfloor} (ne^{-p(\lfloor n/2\rfloor+1)})^i<\frac{2ne^{-p(\lfloor n/2\rfloor+1)}}{(1-p)(1-ne^{-p(\lfloor n/2\rfloor+1)})}=o(1)\quad \text{for } n\to\infty.
$$
\vskip-0.4cm
\end{proof}

\begin{lemma}

Let $\alpha \in (0,1)$ and put $p=n^{-\alpha}$. A.a.s. for any complete $X$ there is an $\bar{X}$-outside-sparse vertex, and there is $K_3$ in $G(n,p)$.
\label{clicks}
\end{lemma}

\begin{proof} By Theorem~\ref{threshold}, a.a.s. any complete subgraph of $G(n,p)$ has at most $\lceil 2/\alpha \rceil +1$ vertices. By Lemma~\ref{extension}, a.a.s., for any $X$ such that $G(n,p)|_X$ is complete, there is an $\bar{X}$-outside-sparse vertex. Finally, by Theorem~\ref{threshold}, a.a.s. there is a copy of $K_3$ in $G(n,p)$.\end{proof}

In Table~\ref{Table}, we give all possible types of pairs of $X$ and $\bar X$ in $G(n,n^{-\alpha})$ that do not have a special vertex, up to a permutation of $X$ and $\bar{X}$.


Each type in the table is assigned a range of values of $\alpha$. Any type appears a.a.s. in $G(n,n^{-\alpha})$ iff it is present in the table and the value of $\alpha$ belongs to the prescribed range. Otherwise, the type a.a.s. do not appear in $G(n,n^{-\alpha})$.

The table reflects the restrictions that are imposed on the types of pairs by Figure~\ref{fig2} and Lemmas~\ref{no_dense}, \ref{common_vertices} and \ref{clicks}. It is not difficult, although a bit tedious, to verify that the 26 cases of the table are\textit{ exactly }the ones that are left after applying the aforementioned statements.

The ranges of $\alpha$ are, however, unexplained in many cases. Therefore, we are left to show that, first, outside the specified ranges of $\alpha$ the corresponding types a.a.s. do not appear in $G(n,n^{-\alpha})$, and that, second, inside the ranges they a.a.s. do appear. 
\renewcommand{\arraystretch}{1.1}
\begin{table}[!ht]
\centering
\begin{tabular}{|c|c|c|c|}

  \hline

  & Type of subsets $X$ $|$ $\bar{X}$   &  Outside types of vertices in $X$ $|$ $\bar{X}$  & $\alpha$  \\

  \hline

  1 & complete $|$ common  & common $|$ dominating, isolated, common  & $(0,1)$ \\

  2 & complete $|$ common  & common $|$ isolated, common  & $(0,1)$ \\

  \hline

  3 & dense $|$ common  & common $|$ dominating, isolated, common  & $(0,4/5)$ \\

  4 & dense $|$ common  & common $|$ isolated, common  & $(0,1)$ \\

  5 & dense $|$ common  & common $|$ dominating, common  & $(0,1/3)$ \\

  6 & dense $|$ common  & common $|$ common  & $(0,1/2)$ \\

  \hline

  7 & independent $|$ common  & common $|$ dominating, isolated, common  & $(0,1)$ \\

  8 & independent $|$ common  & common $|$ isolated, common  & $(0,1)$ \\

  9 & independent $|$ common  & common $|$ dominating, common  & $(0,1/3)$ \\

  10 & independent $|$ common  & common $|$ common  & $(0,1)$ \\

  \hline

  11 & independent $|$ sparse  & common $|$ isolated, common  & $[2/3,1)$ \\

  12 & independent $|$ sparse  & common $|$ common  & $[2/3,1)$ \\

  \hline

  13 & sparse $|$ sparse  & isolated, common $|$ isolated, common  & $[1/2,1)$ \\

  14 & sparse $|$ sparse  & isolated, common $|$ common  & $[1/2,1)$ \\

  15 & sparse $|$ sparse  & common $|$ common  & $[1/2,1)$ \\

  \hline

  16 & sparse $|$ common  & isolated, common $|$ isolated, common  & $[1/2,1)$ \\

  17 & sparse $|$ common  & isolated, common $|$ common  & $[1/2,1)$ \\

  18 & sparse $|$ common  & common $|$ common  & $(0,1)$ \\

  19 & sparse $|$ common  & common $|$ isolated, common  & $(0,1)$ \\

  20 & sparse $|$ common  & common $|$ dominating, common  & $(0,1/2)$ \\

  21 & sparse $|$ common  & common $|$ dominating, isolated, common  & $(0,1)$ \\

  \hline

  22 & common $|$ common  & isolated, common $|$ isolated, common  & $[1/2,1)$ \\

  23 & common $|$ common  & common $|$ common  & $(0,1)$ \\

  24 & common $|$ common  & common $|$ isolated, common  & $(0,1)$ \\

  25 & common $|$ common  & common $|$ dominating, common  & $(0,1/2)$ \\

  26 & common $|$ common  & common $|$ dominating, isolated, common  & $(0,5/6)$ \\

  \hline

\end{tabular}
\vspace{-0.2cm}

\caption{\small Types of pairs of subsets without special vertices.}
\label{Table}
\end{table}
\vskip+0.2cm

We will frequently use the following corollary of Lemma \ref{no_neighb_ext}:
\begin{corol}\label{common_neighbors2} Fix $\alpha\in [1/2,1)$. A.a.s. for any vertex $v$ there are two vertices $u,w$, such that $u\sim v, w\nsim v$ and $N(v)\cap N(u)=\varnothing, N(v)\cap N(w) = \varnothing$.
\end{corol}
To obtain this corollary, we apply Lemma \ref{no_neighb_ext} twice: first, with $m=1$, $k=1$ and $v,u$ playing the roles of $x_1,z$, respectively, and, second, with $m=1$, $k=0$ and $v,w$ playing the roles of $x_1,z$, respectively.

It is easy to see that in all cases when the range of a type is $\alpha\in[1/2,1)$, the nonexistence of the type outside the range is explained using Lemma \ref{no_dense}: for $\alpha<1/2$, a.a.s. there are no two vertices $u,v$, such that $u$ is $X$-isolated and $v$ is $\bar X$-isolated.  Similarly, if the range for a given type is $(0,1/2)$, then the nonexistence of this type for $\alpha\ge 1/2$ is explained by Corollary \ref{common_neighbors2}: for each type with the admissible range $\alpha\in (0,1/2)$, we have a $X$-dominating vertex, and thus we must have an $\bar X$-outside-isolated vertex, guaranteed by Corollary \ref{common_neighbors2}.

For all types we make use of Lemma~\ref{common_vertices}, that guarantees the presense of $X$-outside-common and $\bar X$-outside-common vertices in all cases. In what follows we do not repeat this reference.\\


{\bf Types 1, 2.}

In these cases we have to prove the a.a.s. existence only. Take as $X$  an edge of a triangle for type 1 and a maximal clique for type 2. By Lemma~\ref{clicks}, both types a.a.s. exist in $G(n,p)$.\\

{\bf Types 3 -- 6.}

For these cases we need the following auxiliary lemmas.

\begin{lemma}
Fix $\alpha\in (0,1)$ and put $p=n^{-\alpha}$. A.a.s. in $G(n,p)$ for any $2\lfloor n^{\alpha}\ln \ln n\rfloor$ vertices there is a vertex which is not adjacent to any of them. Fix $2\lfloor n^{\alpha}\ln n\rfloor$ vertices. A.a.s. any other vertex is adjacent to at least one of them and not adjacent to at least one of them.
\label{no_sparse}
\end{lemma}

\begin{proof} Set $m=2\lfloor n^{\alpha}\ln \ln n\rfloor$. Consider distinct vertices $v_1,\ldots,v_m$. Note that we have $(1-p)^m\ge e^{-2mp}\ge (\ln n)^4$. With probability $$
\left(1-(1-p)^{m}\right)^{n-m}\leq e^{-(n-m)/(\ln n)^4}
$$
each vertex $v\in V_n\setminus\{v_1,\ldots,v_m\}$ is adjacent to some of $v_1,\ldots,v_m$. Therefore, for any $m$ vertices there is a vertex which is not adjacent to any of them with the probability at least $1-n^me^{-(n-m)/(\ln n)^4}=1 +o(1)$.

Set $m=2\lfloor n^{\alpha}\ln n\rfloor$ and choose distinct vertices $v_1,\ldots,v_m$. Any other vertex is adjacent to at least one of $v_1,\ldots,v_m$ and not adjacent to at least one of them with the probability
$$
 \left(1-(1-p)^{m}-p^{m}\right)^{n-m}\geq e^{-(1+o(1))/n} = 1+o(1).
$$\vskip-0.4cm
\end{proof}

\begin{lemma}

Let $\alpha \in [1/3,1)$ and put $p=n^{-\alpha}$. A.a.s. $G(n,p)$ does not contain a complete bipartite graph $K_{2,\lceil n^{1/3}\ln\ln n\rceil}$ as a subgraph.
\label{big_bipartite}
\end{lemma}
\begin{proof} Let $m=\lceil n^{1/3}\ln\ln n\rceil$. Such a subgraph exists with a probability at most
$$
 {n\choose 2}{n-2\choose m}p^{2m}\le \frac{n^{m+2} e^m}{m^m}p^{2m}=e^{m(\ln n(1-2\alpha)+1-\ln m)+2\ln n}\leq e^{m(1-\ln\ln\ln n)+2\ln n} = o(1).
$$
\vskip-0.4cm
\end{proof}

\begin{lemma}

Let $\alpha \in (0,1/3)$ and put $p=n^{-\alpha}$. A.a.s. in $G(n,p)$ any two vertices have at least $\lceil n^{1/3}\rceil$ common neighbors.
\label{common_neighbors}
\end{lemma}
\begin{proof} For fixed two vertices $u,v$, denote by $X_i(u,v)$ the indicator of the event $i\sim u$, $i\sim v$. From the Chernoff bound,
$$
 {\sf P}\left(\exists u,v\,\left(\sum_{i=1}^n X_i(u,v)<n^{1/3}\right)\right)\leq n^2 e^{-\frac{\left((n-2)p^2-n^{1/3}\right)^2}{2(n-2)p^2}} = o(1).
$$
\vskip-0.4cm
\end{proof}

We go on to the analysis of types 3 -- 6. If $\alpha<4/5$, then by Theorem~\ref{threshold} a.a.s. $G(n,p)$ contains an induced subgraph on $4$ vertices and $5$ edges: the vertices $v_1,\ldots,v_4$, and all pairs of vertices are adjacent except $v_1,v_2$. Put $X:=\{v_1,v_2,v_3\}$. It has a dominating vertex $v_4$. By Lemma~\ref{extension}, a.a.s. $X$ has type 3.

Any dense, but not complete, set $X$ has cardinality at least 3. Moreover, if it has an $\bar X$-outside-dominating vertex $v$, then the induced graph on $X\cup\{v\}$ has density at least $5/4$. Therefore, if $\alpha>4/5$, then by Theorem~\ref{threshold} a.a.s. in $V_n$ there is no such $X$, and, consequently, no pair of type 3.

Let $s=\lfloor 1/\alpha\rfloor$. For any $\alpha\in(0,1)$, by Lemma~\ref{maximal_extension}
 a.a.s. in $G(n,p)$ there is $X\subset V_n$ with $|X|=s+1,$ such that $G(n,p)|_X$ is a star with no  $\bar{X}$-outside-dense vertices. Lemma~\ref{extension} guarantees that it has an $\bar X$-outside-sparse vertex, and so $X$ is of type 4.

If $\alpha<1/2$, then by Lemma~\ref{degrees} a.a.s. for $v\in V_n$ we have $|N(v)|>2\sqrt n$. Choose $X=\{v\}\cup X_0$, where $X_0\subset N(v)$, $|X| = \lfloor \sqrt n\rfloor$. By Lemma~\ref{no_sparse}, a.a.s. any vertex of $\bar{X}$ is $\bar{X}$-outside-common, so $X$ is a.a.s. of type 6.

By Lemmas~\ref{no_sparse},~\ref{big_bipartite}, if $\alpha\ge 1/3$, then in $G(n,p)$ there are no subsets of type 5 (cf. Table \ref{Table}).  Let $\alpha<1/3$. Fix an ordering on the set of pairs of vertices of $V_n$ and take the first edge $uv$ of $G(n,p)$ in this ordering. By Lemma~\ref{common_neighbors}, a.a.s. $u$ and $v$ share at least $n^{1/3}$ neighbors. Denote $X:=\{u\}\cup \big(N(u)\cap N(v)\big)$. By Lemma~\ref{no_sparse}, a.a.s., all vertices in $\bar X\setminus\{v\}$ are $\bar X$-outside-common. Moreover, by Lemma \ref{no_dense} neither $v$ nor $u$ is $X$-special. Therefore, a.a.s. $X$ has type 5.\\


{\bf Types 7 -- 12.}

This is the most difficult situation. For these cases we need the following auxiliary lemmas.

\begin{lemma}
Consider the event $A$ ``there exists an independent set $X$ such that there is an $\bar{X}$-outside-dominating vertex but there are no $\bar{X}$-outside-isolated vertices''. If $\alpha\in  [1/3,1)$, then a.a.s. $A$ does not hold. If $\alpha\in (0,1/3)$, then a.a.s. $A$ holds.
\label{no_existence}
\end{lemma}

\begin{proof} Let $A_k$ be the number of pairs $(v,X)$, where $X$ is an independent set of size $k$, such that there are no $\bar{X}$-outside-isolated vertices, and $v\in \bar X$ is $X$-dominating. By Lemma~\ref{no_sparse}, it is sufficient to restrict our attention to the case  $k\in I:= [2 n^{\alpha}\ln\ln n, 2n^{\alpha}\ln n]$. For each $k$ we have
\begin{equation}\label{eq1}{\sf E} A_k =  {n\choose k}(n-k) p^{k}(1-p)^{{k\choose 2}}\big(1-(1-p)^k\big)^{n-k-1}.\end{equation}
Let us first prove that $\sum_{k\in I} {\sf E} A_k\to 0$ if $\alpha\ge \frac 1 3$. This will obviously imply the first part of the lemma. Choose $\alpha\ge \frac 13$. Assuming that $k\in I$ and putting $k = n^{\alpha}x\ln n$, we get
$${\sf E} A_k\le \Big(\frac{ne}k\Big)^k n p^k e^{-p{k\choose 2}}\big(1-e^{-pk-p^2k}\big)^{n-k-1} = e^{k(\ln n-\ln k-\ln p -\frac 12pk)-(1+o(1))ne^{-pk}+O(k)}\leq$$
$$e^{k(1-2\alpha-\frac x2)\log n-k\ln\ln n-k\ln x- (1+o(1))n^{1-x}+O(k)}:=f(k).$$
If for some constant $c$ we have $1-x\ge c>\alpha$,  then, obviously, $f(k) = o(1/n)$. Therefore, we may assume that $x\ge 1-\alpha +o(1)$. In that case we have $1-2\alpha-\frac x2 \le \frac 12 -\frac 32\alpha+o(1)$, and thus $f(k) = o(1/n)$ if $\alpha >\frac 13$. If $\alpha=\frac 13$, then $1-2\alpha-\frac x2  = \frac 13-\frac x2$, so $f(k)=o(1/n)$ if $x\ge \frac 23$. Therefore, we assume that $\alpha = \frac 13$ and $x = \frac 23+o(1).$

If $x\le \frac 23-2\frac {\ln\ln n}{\ln n}$, then
$$
k\left(\frac 13-\frac x2\right)\ln n- (1+o(1))n^{1-x}=n^{1/3}x\left(\frac 13-\frac x2\right)(\ln n)^2-(1+o(1))n^{1/3}(\ln n)^2< 0,$$
so in this case $f(k)\le e^{-k\ln\ln\ln n+O(k)} = o(1/n)$.

If $\frac 23-\frac 32\frac {\ln\ln n}{\ln n}\ge x> \frac 23-2\frac {\ln\ln n}{\ln n}$, then $k\big(\frac 13-\frac x2\big)\ln n< k\ln\ln n,$  so $f(k)\le e^{-(1+o(1))n^{1-x}+O(k)} = o(1/n)$, since $k = o(n^{1-x})$.

If $x>\frac{2}{3}-\frac 32\frac {\ln\ln n}{\ln n},$ then $k\big(\frac 13-\frac x2\big)\ln n< \frac 34k\ln\ln n,$ so $f(k)\le e^{-\frac 14k\log\log n+O(k)}=o(1/n)$. Summing over $k\in I$ the bounds that we obtained, we get that $\sum_{k\in I} {\sf E} A_k =o(1)$ for $\alpha\ge 1/3$.\\

Let $\alpha<1/3.$ We prove that ${\sf P}(A_k>0)\rightarrow 1$ as $n\to\infty$, where $k=(1-\alpha)n^{\alpha}\ln n$. This follows from Chebyshev's inequality, since ${\sf E} A_k\to\infty$ as $n\to\infty$ and ${\sf D} A_k=o(({\sf E} A_k)^2)$. The expectation is calculated in (\ref{eq1}). Moreover, the upper bound on ${\sf E} A_k$, proved after (\ref{eq1}), is also a lower bound, if one adds a factor $(1+o(1))$ into the exponent (note that $x=1-\alpha$ in our case):
$$
{\sf E} A_k\ge e^{(1+o(1))\big(k(1-2\alpha-\frac {1-\alpha}2)\ln n-k\ln\ln n- (1+o(1))n^{\alpha}+O(k)\big)} =  e^{(\frac 12+o(1))(1-3\alpha)k\ln n}.
$$
This expression obviously tends to infinity. Next, we show that ${\sf E} A_k^2\le (1+o(1))({\sf E} A_k)^2$. We have
$${\sf E} A_k^2\le \sum_{j=0}^k{n\choose k}{n-k\choose k-j}{k\choose j}(n-k)^2p^{2k-j}(1-p)^{2{k\choose 2}-{j\choose 2}}\big(1-2(1-p)^k+(1-p)^{2k-j}\big)^{n-2k-2}.$$
Let us denote the $j$-th summand of the above expression by $g(j)$.
It is easy to see that $g(0) = (1+o(1))({\sf E}A_k)^2$. Indeed,
$$
\frac {g(0)}{({\sf E} A_k)^2}\le (1-(1-p)^k)^{-4k-4} \le \exp \big(O(k)e^{-pk(1+o(1))}\big)=\exp\big(n^{-(1+o(1))(1-2\alpha)}\big)=1+o(1).
$$
Thus, it is sufficient for us to show that for each $j\ge 1$ the expression $g(j)$ is much smaller than $g(0)$.\\

Let us first consider the case $j\le \alpha k$. Then
$$\frac{\big(1-2(1-p)^k+(1-p)^{2k-j}\big)^{n-2k-2}}{\big(1-2(1-p)^k+(1-p)^{2k}\big)^{n-2k-2}}\le \big(1+(1+o(1))(1-p)^{2k-j}\big)^{n}\le \big(1+(1+o(1))e^{-(2-\alpha)kp}\big)^{n}.$$
We have $(2-\alpha)kp = (2-\alpha)(1-\alpha) \ln n>c\log n,$ where $c>1$. Consequently, for $n$ large enough, the last expression in the displayed formula is at most $\big(1+\frac 1{n^c}\big)^n = 1+o(1)$. Therefore, in this case,
$$
\frac{g(j)}{g(0)} \sim \frac{{n-k\choose k-j}{k\choose j}}{{n\choose k}p^j(1-p)^{{j\choose 2}}}=O\Big(\frac{k^j \big(\frac {ke}j\big)^j}{n^j p^{j}e^{-(\frac 12+o(1))pj^2}}\Big)\leq\frac{1}{n^{(1-3\alpha+o(1))j}j^je^{-(\frac 12+o(1))pj^2}}.
$$
Since $j\le \alpha k=\alpha(1-\alpha)n^{\alpha}\ln n$, we have $j^je^{-(\frac 12+o(1))pj^2}=e^{j(\log j-(\frac 12+o(1))pj)} = \Omega(1)$. Therefore, for each $j\le \alpha k$, we have $\frac{g(j)}{g(0)}\leq n^{-(1-3\alpha+o(1))j}$.\\

Next, consider the case $j\ge \alpha k$. In this case, we use the following crude estimate:
$$\frac{\big(1-2(1-p)^k+(1-p)^{2k-j}\big)^{n-2k-2}}{\big(1-2(1-p)^k+(1-p)^{2k}\big)^{n-2k-2}}\le \big(1+(1+o(1))(1-p)^{k}\big)^{n}\le e^{(1+o(1))ne^{-kp}}=e^{(1+o(1))n^{\alpha}}.$$
Remark that for $j\ge \alpha k$ we have $e^{(1+o(1))n^{\alpha}} = n^{o(j)}$, $j^j = n^{(1+o(1))\alpha j}$. Finally, we have
$$\frac{g(j)}{g(0)} =O\Bigg(\frac{k^j \big(\frac {ke}j\big)^j e^{(1+o(1))n^{\alpha}}}{n^j p^{j}e^{-(\frac 12+o(1))pj^2}}\Bigg)=\frac{1}{n^{(1-2\alpha+o(1))j} e^{-(\frac 12+o(1))pj^2}}.$$
Since $j\le k$, we have
$$
e^{-(\frac 12+o(1))pj^2}\ge e^{-(\frac 12+o(1))(1-\alpha)j\ln n} = n^{-\big(\frac 12-\frac{\alpha}2+o(1)\big)j}.
$$
Therefore, $\frac{g(j)}{g(0)} = n^{-\big(\frac 12-\frac{3\alpha}2+o(1)\big)j}$ for each $j\ge \alpha k$. We conclude that there exists $\delta>0$ such that $\sum_{j=1}^k g(j)\leq\sum_{j=1}^k n^{-\delta j}g(0)=o(g(0))$, and so ${\sf E} A_k^2\le (1+o(1))({\sf E} A_k)^2$.
\end{proof}

\begin{lemma}
If $\alpha\in (0, 2/3)$, then a.a.s. there is no vertex $v$ with $N(v)$ being an independent set. If $\alpha\in [2/3,1)$, then a.a.s. there is such a vertex $v$.
\label{independent_neighbors}
\end{lemma}
\begin{proof} Let $\alpha<2/3$. Fix a vertex $v\in V$. Let $X$ be a set of neighbors of $v$ in $G(n,p)$. The probability $\gamma=\gamma(X)$ of $X$ being an independent set equals
$$
 \sum_{k=0}^{n-1} {n-1\choose k} p^k (1-p)^{n-1-k+{k\choose 2}}\leq \sum_{k=0}^{n-1}\Big(\frac{enp(1-p)^{\frac nk}}k\Big)^k(1-p)^{\frac {k^2}2-\frac{3k}{2}-1}\le
$$
$$
\sum_{k=0}^{n-1}\Big(\frac{enp}ke^{-\frac {np}k}\Big)^k(1-p)^{\frac {k^2}2-\frac{3k}{2}-1}\le \sum_{k=0}^{\lceil p^{-1/2}\log ( 1/p)\rceil}\Big(\frac{enp}ke^{-\frac {np}k}\Big)^k+\sum_{k=\lceil (p^{-1/2}\log (1/p)\rceil}^{n-1}(1-p)^{\frac {k^2}2-\frac{3k}{2}-1}.
$$
In the last transition we used the fact that $ex\le e^x$ for any $x$. We also have $ex\le e^{(\frac 12+o(1))x}$, and thus the first sum in the last displayed expression is bounded by $ne^{-\Omega(np^{3/2}/\ln (1/p)})\le n e^{-n^{\beta}}$ for some $\beta>0$, since $\alpha<2/3$. This expression is obviously $o(1/n)$. The second sum in the last displayed equation is $O\big(ne^{-p\big(p^{-1}(\ln (1/p))^2\big)}\big) = o(1/n)$, since $1/p = n^{\alpha}$ for some $\alpha>0$. Therefore, a.a.s. there is no vertex $v$ in $V_n$ with $N(v)$ being an independent set.\\

If $\alpha>2/3$, then
$$
\gamma\geq\sum_{k=\lfloor 1/2 n^{1-\alpha}\rfloor}^{\lfloor 3/2 n^{1-\alpha}\rfloor}{n-1\choose k} p^k (1-p)^{n-1-k+{k\choose 2}}=
(1+o(1))\sum_{k=\lfloor 1/2 n^{1-\alpha}\rfloor}^{\lfloor 3/2 n^{1-\alpha}\rfloor}{n\choose k} p^k (1-p)^{n-k}=1+o(1).
$$
If $\alpha=2/3$, then
$$
\gamma\geq (1+o(1))\sum_{k=\lfloor 1/2 n^{1/3}\rfloor}^{\lfloor 3/2 n^{1/3}\rfloor}{n\choose k} p^k (1-p)^{n-k+k^2/2}=(1-p)^{9/4 n^{2/3}}(1+o(1))\rightarrow e^{-9/4},\quad n\rightarrow\infty.
$$

Lemma \ref{no_neighb_ext} implies that for each $k=2,3,\ldots$ we have a set of vertices $v_1,\ldots, v_k$ such that $v_i$'s form and independent set and, moreover, $N(v_i)\cap N(v_j)=\varnothing$ for each $1\le i<j\le k$. The event ``$N(v_i)$ is an independent set'' is independent of events ``$N(v_j)$ is an independent set'' for $j\ne i$. Therefore, for any $\varepsilon>0$ there exists $k$ such that
$$
 {\sf P}(\exists m\in\{1,\ldots,k\}\,X_m\text{ is an independent set})\geq (1-\varepsilon)+o(1).
$$
This means that a.a.s. in $G(n,p)$ there exists $v$ such that $N(v)$ is an independent set. \end{proof}

\begin{lemma}
Let $\alpha \in (0,1)$ and put $p=n^{-\alpha}$. A.a.s. in $G(n,p)$ there exists an independent set $X$ such that $\bar{X}$ is a common set, and all its vertices are $\bar{X}$-outside-common.
\label{existence}
\end{lemma}
\begin{proof} If $\alpha<2/3$, then the statement follows from Lemmas~\ref{maximal_extension}~and~\ref{independent_neighbors}. Let $\alpha\geq 2/3$. The  cardinality $N$ of the largest independent set in $G(n,p)$ is given in Lemma \ref{ind_set}. Moreover, by Hoeffding-Azuma inequality, $\frac{N-{\sf E}N}{n^{1/2+\delta}}\stackrel{{\sf P}}\to 0$ for any $\delta>0$. In particular, ${\sf E}N\sim 2(1-\alpha)n^{\alpha}\ln n$.

We construct a maximal independent set $X$ by adding vertices step by step in the following way. At the first step, we put $v_1:=1$ into $X$ and remove $1$ and all neighbors of $1$ from $V$. At the $i$-th step  ($i\geq 2$), we put $v_i:=\min \{j:j\in V\}$ into $X$ and remove $v_i$ and all its neighbors from $V$. At the final step, we get $X=\{v_1,\ldots,v_{|X|}\}$ and $V=\varnothing$.

Consider the event $A$ that there is a vertex $v$ outside $X$ with $N(v)\subset X$. For $i\in\{1,\ldots,|X|\}$, let $A_i$ be the event ``there is a neighbor of $v_i$ such that all its neighbors are in $X$''. We clearly have $A=\cup A_i$. For each step $i$ and $v\in N(v_i)$ all edges between $v$ and $X_i:=X\setminus\{v_1,\ldots,v_i\}$ appear mutually independently with probability $p$.
For any $c>0$ and $n$ large enough we have
$$
{\sf P}\big(\exists u\,\,(u\sim v_i)\wedge(N(u)\subset X)\big)\leq n{\sf P}(N(v)\subset X)\leq
$$
$$
n\left[1-{\sf P}\left(|N(v)\cap X_i|<(c+1)p|X_i|,|N(v)|>\frac{1}{2}n^{1-\alpha},|X_i|<2n^{\alpha}\ln n\right)\right].
$$
Let $c=\frac{8}{1-\alpha}$.
By the Chernoff bound,
$$
{\sf P}(|N(v)\cap X_i|\geq (1+c)p|X_i|)\leq e^{-\frac{3}{1-\alpha}{p|X_i|}}\quad \text{and}\quad
{\sf P}\left(|N(v)|\leq\frac{1}{2}n^{1-\alpha}\right)\leq e^{-\frac{1}{8}n^{1-\alpha}}.
$$
By Hoeffding-Azuma inequality,
$$
{\sf P}\left(|X|\geq 2n^{\alpha}\ln n\right)\leq e^{-(4\alpha)^2(\ln n)^2n^{2\alpha-1}(1+o(1))},
$$
$$
{\sf P}\left(|X|< (1-\alpha)n^{\alpha}\ln n\right)\leq e^{-(1-\alpha)^2(\ln n)^2n^{2\alpha-1}(1+o(1))}.
$$
We have the following chain of inequalities:
$$
1-{\sf P}\left(|N(v)\cap X|<2p|X_i|,|N(v)|>\frac{1}{2}n^{1-\alpha},|X_i|<2n^{\alpha}\ln n\right)\leq
$$
$$
{\sf P}\left(|X_i|\geq 2n^{\alpha}\ln n\right)+
{\sf P}\left(|N(v)|\leq\frac{1}{2}n^{1-\alpha}\right)+
{\sf P}\left(|X|< (1-\alpha)n^{\alpha}\ln n\right)+
$$
$$
{\sf P}\left(\left.|N(v)\cap X_i|\geq \left(1+\frac{8}{1-\alpha}\right)p|X_i|\right||X_i|\geq (1-\alpha)n^{\alpha}\ln n\right){\sf P}\left(|X_i|\geq (1-\alpha)n^{\alpha}\ln n\right)\leq
$$
$$
e^{-(4\alpha)^2(\ln n)^2n^{2\alpha-1}(1+o(1))}+
e^{-\frac{1}{8}n^{1-\alpha}}+e^{-(1-\alpha)^2(\ln n)^2n^{2\alpha-1}(1+o(1))}+e^{-3\ln n},
$$
where the last inequality holds due to the independence of $N(v)$ and $X_i$. Finally, we get
$$
{\sf P}(A)\leq\sum_{i=1}^n{\sf P}(A_i)\leq
n^2 e^{-3\ln n(1+o(1))}\to 0.
$$
\vskip-0.4cm
\end{proof}

We go on to the analysis of types 7 -- 12.  By Theorem~\ref{threshold} (note that it is also holds for induced subgraphs), a.a.s. there are $v_1,v_2,v_3\in V_n$ such that $v_1\sim v_2,v_2\sim v_3,v_1\nsim v_3$. Put $X:=\{x_1,x_3\}$. By Lemma~\ref{no_dense}, a.a.s. the set $\bar X$ is common. By Lemmas~\ref{degrees},~\ref{extension}, a.a.s. $X$ has type 7.

Let $s=\lfloor 1/\alpha\rfloor+1$. By Lemma~\ref{maximal_extension}, there exists an independent set $X\subset V$ such that $|X|=s$ and there are no common neighbors of vertices of $X$ in $G(n,p)$. By Lemma~\ref{degrees}, the set $\bar{X}$ is common. By Lemmas~\ref{degrees},~\ref{extension}, a.a.s. the set $X$ has type 8.

By Lemma~\ref{no_existence}, if $\alpha\geq 1/3$, then a.a.s.  there is no sets of type 9 in $G(n,p)$. If $\alpha<1/3$, then by Lemmas~\ref{independent_neighbors} and \ref{no_existence} a.a.s. $G(n,p)$ contains a set of type 9.

The set $X$ of type 10 exists by Lemmas \ref{sizes} and \ref{existence}.

If $\alpha<2/3$, then a.a.s. in $G(n,p)$ there are no sets of types 11, 12 by Lemma~\ref{independent_neighbors}.
If $\alpha\geq 2/3$, then, by Lemmas~\ref{independent_neighbors}~and~\ref{no_sparse}, in $G(n,p)$ there is a set $X$ of type 12. Obviously, a.a.s. it has a subset of type 11 as well.\\

{\bf Types 13 -- 26.}

Let $\alpha\ge 1/2$. Applying Lemma \ref{no_neighb_ext} as in the proof of Lemma \ref{independent_neighbors}, we get that there are four vertices $v_1,\ldots, v_4$, that form an independent set, and such that no two of them share a common neighbor. Split $V_n\setminus (\cup_{i=1}^4 N(v_i))$ randomly into two almost equal parts $V_1$ and $V_2$. By Lemma~\ref{degrees}, a.a.s. any vertex $v\in V_n\setminus\{v_1,v_2,v_3,v_4\}$ has both neighbors and non-neighbors in both $V_1$ and $V_2$.

To prove the a.a.s. existence of type 13, put $X:= \{v_1\}\cup N(v_2)\cup \{v_3\}\cup N(v_3)\cup V_1$. One can see that $v_1$ is $X$-inside-isolated, $v_3$ is $X$-outside-isolated, $v_2$ is $\bar X$-inside-isolated, and $v_4$ is $\bar X$-outside-isolated.

For type 14, put $X:=\{v_1,v_2,v_4\}\cup N(v_2)\cup N(v_3)\cup V_1$. For type 15 put $X:= \{v_1,v_3\}\cup N(v_2)\cup N(v_4)\cup V_1$. For type 16 put $X:=\{v_1,v_2\}\cup N(v_2)\cup V_1$. For type 17 put $X:= \{v_1,v_2,v_3,v_4\}\cup N(v_2)\cup V_1$. For type 22 put $X:= \{v_1\}\cup N(v_1)\cup V_1$.\\

The types 18, 19 and 24 are even easier to obtain. Take two nonadjacent vertices $u_1,u_2$, and split the set $V_n\setminus \big(N(u_1)\cup N(u_2)\big)$ into two almost equal parts $U_1,U_2$. Once again, a.a.s., any other vertex has both neighbors and non-neighbors in both $U_1$ and $U_2$.

To obtain type 18, put $X:= \{u_1,u_2\}\cup U_1$. To obtain type 19, put $X:= \{u_1\}\cup U_1$. To obtain type 24, put $X:= U_1$.

To obtain type 23, simply split randomly $V$ into two almost equal parts and choose $X$ to be equal to one of these parts. Then a.a.s. all vertices are $X$-common and $\bar X$-common.\\

We are left to deal with types 20, 21, 25, 26. Let $\alpha<1/2$. Take any vertex $v$ and its neighbor $w$. A.a.s., $|N(w)\cap N(v)| = o(|N(v)|)$. Put $X:= N(v)\setminus N(w)$. Then $X$ is sparse, it has a dominating vertex $v$, which is a.a.s. not special. Moreover, a.a.s. $|X|=\Theta (n^{1-\alpha})$, which by Lemma \ref{no_sparse} means that any vertex from $\bar X\setminus \{v\}$ is $\bar X$-outside-common. By Lemma~\ref{no_dense}, $X$ is of type 20.

To obtain a set of type 25, consider any vertex $v$ and the set of its neighbors $N(v)$. Put $X:=N(v)\setminus \{w\}$, where $w$ is a neighbor of $v$. A.a.s., $\deg_X(v)>0$ for each $v\in V$ by Lemma~\ref{degrees}. This implies that a.a.s. there is no $X$-inside-isolated vertex.    Moreover, there are no $\bar X$-outside-isolated vertices by Lemma \ref{no_sparse}. Therefore, $X$ is a.a.s. of type 25.

Let $\alpha \in (0,1)$. By Theorem~\ref{threshold} and Lemma~\ref{extension}, the graph $G(n,p)$ a.a.s. contains an induced subgraph, isomorphic to a triangle with a hanging edge: $v_1,v_2,v_3,v\in V$, $v_i\sim v_j$, $v\sim v_1, v\nsim v_2,v\nsim v_3.$ Put $X:=\{v,v_2,v_3\}$. Then $X$ is sparse, and it has a dominating vertex $v_1$. Moreover, by Lemma \ref{no_sparse}, a.a.s. there are $\bar X$-outside-isolated vertices and there are no $X$-outside-isolated vertices. That is, $X$ is a.a.s. of type 21.

Let $X$ be of type 26 and let $v$ be $X$-dominating. Then the subgraph, induced on $X\cup\{v\}$, has $|X|+1$ vertex and at least $\lceil\frac 32 X\rceil$ edges (each vertex of $X$ has degree at least 1 in $X$). Moreover, $|X|>3$. It is not difficult to see that there is a subset $X_0$ in $X$ such that $|X_0|=4$ and the graph induced on $X_0\cup\{v\}$ contains two triangles sharing exactly one vertex. Thus, $G(n,p)|_{X_0\cup\{v\}}$ has density at least $6/5$, and by Theorem \ref{threshold}, $G(n,p)$ a.a.s. does not contain such subgraphs for $\alpha>5/6$.

If $\alpha<5/6$, the subgraph (denote it by $G|_X$) described in the previous subgraph do a.a.s. exist in $G(n,p)$. Then $X$ is common, it has a dominating vertex $v$, and a.a.s. it has $\bar X$-outside-isolated vertices. Therefore, $X$ is of type 26. Table \ref{Table} is verified completely.\\

\subsubsection{Special vertices}


Assume that $X$ has a special vertex $v$. If $v$ is special for $X$, it is special for $\bar X$ as well. Therefore, we may assume that $v\in X$.

We have two cases to consider: the case when $v$ is $X$-isolated and $\bar X$-dominating, and the case when $X$ and $\bar X$ are interchanged. By Lemma \ref{no_dense}, there are no other $X$-dominating or $\bar X$-dominating vertices, as well as other special vertices for $X$.

\textbf{Case 1} If $v$ is $X$-isolated and $\bar X$-dominating, then,  by Statement 1 of Lemma \ref{no_dense}, there are no other $X$-isolated vertices. In particular, all vertices of $\bar X$ are $\bar X$-outside-common.

\textbf{Case 2} If $v$ is $X$-dominating and $\bar X$-isolated, then
there are no other $\bar X$-isolated vertices, as well as $X$-inside-isolated vertices.

We summarize the remaining cases in a smaller table, resembling Table \ref{Table}. Remark that the type of $X$ is determined by the type of the special vertex, so it is not listed in Table \ref{Table2}. Moreover, we list outside types of not special vertices only. We remark that in this case the limitations $\alpha\in [1/2,1)$ are explained by the application of Lemma \ref{no_dense}: for $\alpha<1/2$ there are no two vertices $u,v$, such that $u$ is $X$-isolated and $v$ is $\bar X$-isolated.

As in the case of Table \ref{Table}, we need to prove that the types listed in Table \ref{Table2} are a.a.s. present in $G(n,p)$ within the ranges present, and that the other types a.a.s. do not appear. Note that some of the cases have $\varnothing$ as the range of $\alpha$, which means that they a.a.s. do not appear for any $\alpha$. We have left them in the table since they are not ruled out by the previous considerations.

\begin{table}[!ht]
\centering
\begin{tabular}{|c|c|c|c|c|}

  \hline

   &&Type of $\bar{X}$   &  Outside types of vertices in $X$ $|$ $\bar{X}$  & $\alpha$  \\

\hline

  Case 1&1.1 & sparse  & common, isolated $|$ common  & $[1/2,1)$ \\

  &1.2& sparse  & common $|$ common  & $\varnothing$ \\
  \cline{2-4}
&1.3& independent  & common, isolated $|$ common  & $[2/3,1)$ \\

 &1.4 & independent  & common $|$ common  & $\varnothing$ \\
  \cline{2-4}

&1.5& common  & common, isolated $|$ common  & $\varnothing$ \\

  &1.6& common & common $|$ common  & $(0,1/2)$ \\
  \hline
  Case 2 &2.1& common  & common $|$ common, isolated  & $[1/2,1)$ \\

\cline{2-4}
 &2.2 & common  & common $|$ common  & $(0,1/2)$ \\
 \hline

\end{tabular}
\vspace{-0.2cm}

\caption{\small Types of pairs of subsets.}
\label{Table2}
\end{table}
\vskip+0.2cm

Actually, Lemma \ref{no_neighb_ext} is sufficient to explain most of Table \ref{Table2}, via Corollary~\ref{common_neighbors2}.

Choose $\alpha\in [1/2,1)$ and take a vertex $v$ that lies in a triangle guaranteed by Lemma \ref{clicks}. Putting $\bar X:=N(v)$, the corollary above  guarantees an $X$-outside-isolated and an $\bar X$-inside-isolated vertex. Thus, 1.1 is explained, as well as 1.2: type 1.2 is impossible for $\alpha<1/2$ since $\bar X$ cannot be sparse, and it is impossible for $\alpha\ge 1/2$ since then there must be an $X$-outside-isolated vertex. Similarly,  types 1.4 and 1.5 are ruled out as well. The explanation in case 2.1 is the same as in case 1.1.

We are left with types 1.3, 1.6 and 2.2. The existence of 1.6 and 2.2 is simple: since for $\alpha<1/2$ we neither can have $X$-isolated nor $\bar X$-isolated vertices other than the special vertex, taking a vertex $v$ and putting $\bar X:=N(v)$ and $X:=N(v)\cup\{v\}$, we get  types 1.6 and 2.2, respectively. For $\alpha\ge 1/2$, the existence of either $X$-isolated or $\bar X$-isolated vertices is guaranteed by Corollary \ref{common_neighbors2}, so these types do not appear for $\alpha\ge 1/2$.

The a.a.s. existence of type 1.3 within the range $\alpha\in [2/3,1)$ (as well  as its a.a.s. nonexistence outside this range) is explained by Lemma~\ref{independent_neighbors} combined with Corollary \ref{common_neighbors2}.

\subsection{Pairs of subsets for $\alpha>1$}\label{alpha>1}
 We may suppose that $\alpha = 1+\frac 1{\ell}$ for $\ell\ge 7, \ell \ne 8$. Assume that Spoiler chooses a set $X\subset V(G)$ (the case that it chooses $X\subset V (H)$ is symmetric
of course). Then Duplicator chooses a set of vertices $Y\subset V(H)$. It is sufficient to choose $Y$, such that the pair $Y,\bar Y$ is of the same type as $X, \bar X$ (see Section \ref{sec413}). Below, we describe the algorithm which Duplicator utilizes to achieve its goal. For $Z\in\{X,\bar X\}$ put $Y(Z):=Y$ if $Z=X$ and $Y(Z):= \bar Y$ if $Z=\bar X$. In what follows, Duplicator makes each of his steps based on the information about \textit{both} choices of $Z$ \textit{simultaneously}. Recall that we assume that $X$ is nontrivial ($|X|\ge 2$ and $|\bar X|\ge 2$).


First, Duplicator checks, whether
there are some isolated vertices in $Z$. Denote the set of isolated vertices in a graph $W$ by $I(W)$. Duplicator splits $I(H)$ between $Y$ and $\bar Y$ according to the rule: for both choices of $Z$ we have $I(G)\cap Z\ne \varnothing$ iff $I(H)\cap Y(Z)\ne \varnothing$.

Let us call connected components of a graph with at least one edge \textit{nonempty components}. Next, Duplicator checks, whether there is a nonempty component $C$ of $G$ such that $C\cap Z=\varnothing$.
Depending on the outcome, there are the following two cases.\\

\textbf{Case 1} Assume each nonempty component of $G$ intersects both $X$  and $\bar X$. Since a.a.s. $G$ has at least two components, neither there exists $X$-dominating, nor $\bar X$-dominating vertex. Moreover, since $G(n,p)$ a.a.s. contains connected components that are edges,
these edges in $G$ are split between $X$ and $\bar X$, and so both  contain inside-isolated, outside-common vertices.


If neither $X$ nor $\bar X$ contains any inside-common vertices, then basically $X$, $\bar X$ define a proper two-coloring of $G$. In this case Duplicator chooses any proper two-coloring (that is, a coloring of the vertices into two colors without monochromatic edges) of the nonempty components of $H$ and puts $Y$ to be one of these classes (together with the possibly chosen isolated vertices).

Denote a simple path on $n$ vertices  by $P_n$. Then $\rho^{max}(P_n)=1-\frac 1n$.
Assume that $Z$ contains an inside-common, outside-isolated vertex (which implies that $Z$ has an inside-common, outside-common vertex as well, due to the fact that each connected component of $G$ has nonempty intersection with both $X$ and $\bar X$). Then $G$ contains $P_3$, which is only possible if $\ell\geq 2$, and so there are components of $H$ that are isomorphic to $P_3$. In that case Duplicator takes such a path and puts two consecutive vertices into $Y(Z)$, while the third vertex into its complement.

If $Z$ contains an inside-common, outside-common vertex, but no inside-common, outside-isolated vertex, then $G$ a.a.s. contains $P_4$, which is only possible if $\ell\geq 3$, and so (by \textbf{T3}(3)) there are components of $H$ isomorphic to $P_4$. In that case Duplicator takes such a component and puts two inner vertices into $Y(Z)$ and its two leaf vertices into $Y(\bar Z)$.

After these two steps both $Y$ and $\bar Y$ have the same types of inside-common vertices as $X$ and $\bar X$, respectively. The remaining all but at most two nonempty components of $H$ are then properly colored in two colors and split between $Y$ and $\bar Y$. It is not difficult to see that at the end the pair $Y, \bar Y$ has the same type as $X, \bar X$.\\

\textbf{Case 2} Assume that there is a component $C$ such that $C\cap Z=\varnothing$. We may w.l.o.g. assume that $Z=X$. If this holds for both $X$ and $\bar X$, then we assume that $X$ has fewer connected components (which implies that in any case $\bar X$ is not connected). Therefore, $C\subset \bar X$ and each vertex in $C$ is $\bar X$-outside-isolated. This implies that $\bar X$ a.a.s. contains inside-common, outside-isolated vertices (recall that we do not care about inside-isolated, outside-isolated vertices, since they are simply isolated vertices in $G$ and were treated in the beginning of Section~\ref{alpha>1}). Moreover, $\bar X$ does not belong to one connected component, and so does not have dominating vertices. This implies that the only ambiguity concerning the types of vertices in $\bar X$ is which outside-common and which outside-dominating vertices it has.

In what follows, we describe a strategy of Duplicator to construct step by step a forest $T$  and a subset $Y'$ of its vertices such that, first, for some union $F$ of components in $H$ the graph  $H|_F$ is isomorphic to $T$, and, second, for the subset $Y\subset F$ that corresponds to $Y'$, the pair $Y,\bar Y$ is a pair of the same type as $X,\bar X$.


Note that $T$ is  an {\it abstract} graph, which a priori has no relation to $H$. First, Duplicator checks the type of $X$:
\begin{itemize}\setlength\itemsep{0em} \item if $X$ is complete, then Duplicator takes an edge;\item
if $X$ is dense, then Duplicator takes a $P_3$;\item
if $X$ is common, then Duplicator takes two disjoint edges;\item
if $X$ is independent, then Duplicator takes two isolated vertices.\end{itemize}
The vertices produced on this step form the set $Y'$.

Next, Duplicator checks the types of vertices inside $X$:\begin{itemize}\setlength\itemsep{0em} \item
if all the vertices \textit{of a given inside type} in $X$ are outside-common, then he attaches a new edge to each vertex that has this type in $Y'$;\item
if some, but not all of the vertices of a given inside type in $X$ are outside-common, then he attaches a new edge to \textit{exactly one} vertex of $Y'$ of this inside type.\end{itemize}
The vertices produced on this step form the set $A_1$.

Finally, Duplicator checks the outside types of vertices in $\bar X$: \begin{itemize}\setlength\itemsep{0em}
\item
if there is an $X$-dominating vertex in $\bar X$ (which is possible only in the case if $X$ is an independent set), then he adds a vertex and connects it to each of the vertices of $Y'$. He adds an extra vertex and connects it with an edge to the dominating vertex iff it is $\bar X$-inside-common;\item
if all the outside-common vertices in $\bar X$ are inside-common, then he attaches a new edge to each vertex from $A_1$;\item
if some outside-common vertices in $\bar X$ are inside-common and some are inside-isolated, then he attaches a new edge to one (no matter which one) vertex in $A_1$.\end{itemize}
The vertices added at this step form the set $A_2$.
In the last step we may run into a small trouble: if $|A_1|=1$, then we cannot have both types of outside-common vertices in $A_1$. In this case we check the type of the vertex in $Y'$ that is connected to the only vertex in $A_1$ and connect a new edge to \textit{a vertex $v\in Y'$ that has the same type and lies in (one of) the smallest component of the graph induced on the set $Y'\cup A_1\cup A_2$}. The set added at this ``trouble'' step we call $A_3$. Note that $|A_3|\le 1$. Note that if $|A_1|=0$, then there are no outside-common vertices in $\bar X$.\\
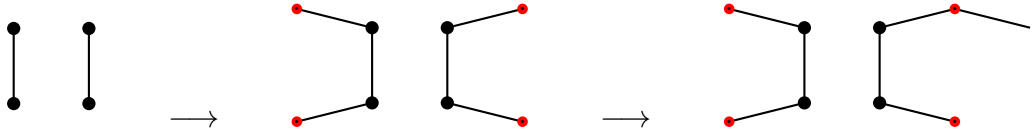
\begin{figure}
\begin{center}
\begin{tikzpicture}[scale=0.5]


\draw[thick] (-1,-1) -- (-1,1);
\draw[thick] (1,-1) -- (1,1);
\node[fill=black, circle, inner sep = -0pt, minimum size=5pt] at (-1,-1) {.};
\node[fill=black, circle, inner sep = -0pt, minimum size=5pt] at (1,-1) {.};
\node[fill=black, circle, inner sep = -0pt, minimum size=5pt] at (-1,1) {.};
\node[fill=black, circle, inner sep = -0pt, minimum size=5pt] at (1,1) {.};
\node[white,fill=white, circle, inner sep = -0pt, minimum size=0pt] at (-1,-1.5) {.};

\end{tikzpicture}
\ \ \ \ \ \ $\longrightarrow$ \ \ \ \ \ \ \begin{tikzpicture}[scale=0.5]


\draw[thick] (-3,-1.5)--(-1,-1) -- (-1,1)--(-3,1.5);
\draw[thick] (3,-1.5)--(1,-1) -- (1,1)--(3,1.5);
\node[fill=black, circle, inner sep = -0pt, minimum size=5pt] at (-1,-1) {.};
\node[fill=black, circle, inner sep = -0pt, minimum size=5pt] at (1,-1) {.};
\node[fill=black, circle, inner sep = -0pt, minimum size=5pt] at (-1,1) {.};
\node[fill=black, circle, inner sep = -0pt, minimum size=5pt] at (1,1) {.};
\node[fill=red, circle, inner sep = -0pt, minimum size=4pt] at (-3,-1.5) {.};
\node[fill=red, circle, inner sep = -0pt, minimum size=4pt] at (3,-1.5) {.};
\node[fill=red, circle, inner sep = -0pt, minimum size=4pt] at (-3,1.5) {.};
\node[fill=red, circle, inner sep = -0pt, minimum size=4pt] at (3,1.5) {.};

\end{tikzpicture}
\ \ \ \ \ \ $\longrightarrow$ \ \ \ \ \ \ \begin{tikzpicture}[scale=0.5]


\draw[thick] (-3,-1.5)--(-1,-1) -- (-1,1)--(-3,1.5);
\draw[thick] (3,-1.5)--(1,-1) -- (1,1)--(3,1.5)--(5,1);
\node[fill=black, circle, inner sep = -0pt, minimum size=5pt] at (-1,-1) {.};
\node[fill=black, circle, inner sep = -0pt, minimum size=5pt] at (1,-1) {.};
\node[fill=black, circle, inner sep = -0pt, minimum size=5pt] at (-1,1) {.};
\node[fill=black, circle, inner sep = -0pt, minimum size=5pt] at (1,1) {.};
\node[fill=red, circle, inner sep = -0pt, minimum size=4pt] at (-3,-1.5) {.};
\node[fill=red, circle, inner sep = -0pt, minimum size=4pt] at (3,-1.5) {.};
\node[fill=red, circle, inner sep = -0pt, minimum size=4pt] at (-3,1.5) {.};
\node[fill=red, circle, inner sep = -0pt, minimum size=4pt] at (3,1.5) {.};
\end{tikzpicture}
\end{center} \caption{An example of the three stages of construction of $T$. A set $X$ is common. All inside-common vertices in $X$ are outside common. Some $\bar X$-outside-common vertices are $\bar X$-inside common, and some are $\bar X$-inside isolated. The vertices of $X$ are marked by big black circles, and the vertices of $A_1$ are marked by smaller red circles.}\label{fig4}
\end{figure}
Let $T$ be the forest on $Y'\cup A_1\cup A_2\cup A_3$ with the edges described above. It is clear by the definition of $T$ that, if a set of connected components $C$ in $H$ is isomorphic to $T$ and $Y$ is a subset of $V(H)$ that corresponds to $Y'$, then the pair $Y,\bar Y$ has the same type as $X,\bar X$. Note that at each step we increased the size of the components by the least possible value. This means that in $G$ there is a connected component $F$ such that $v(F)\geq v(T')$, where $T'$ is the largest component of $T$. So, $\ell\geq v(T')-1$. Lastly, it is not difficult to check that $v(T')$ is either at most $7$, or equal to $9$. We note that $9$ is only possible if we start with $P_3$, attach one edge to each of the vertices of $X$ on the second step, and attach one more edge to the $A_1$-vertices on the last step. This leads to the tree that is depicted on Figure~\ref{figt}.  In this case, $T$ is a tree, and so $T=T'$. 

It remains to prove that, a.a.s., in $H$, there exists an induced copy of $T$ which is a union of components in $H$. For $\ell\ge 9$, the property T3$(\ell-1)$ implies that a.a.s., for every given tree of size at most 9, there exists a connected component of $H$ isomorphic to this tree. This immediately implies the a.a.s. existence of a desired copy of $T$. If $\ell=7$, then by T2$(\ell)$ a.a.s. there is no component of size at least 9 in $G$, and thus $T'$ has at most $7$ vertices. The property T3$(\ell-1)$, in turn, implies that a.a.s. there exists a connected component of $H$ isomorphic to $T'$ as well as components isomorphic to all the other trees in $T$. Therefore, a desired set $Y$ exists.\\

{\sc Acknowledgements:} We would like to thank the anonymous referee for numerous useful comments on the presentation of the paper.


\begin{thebibliography}{99}
\bibitem{AS} N. Alon, J.H. Spencer, {\it The Probabilistic Method}, John Wiley \& Sons, 2000.


\bibitem{Bollobas} B. Bollob\'{a}s, {\it Random Graphs}, 2nd Edition, Cambridge University Press, 2001.

\bibitem{Bol_small_thrshld} B. Bollob\'{a}s, {\it Random Graphs}, Proc. in Combinatorics, Swansea 1981, London Math. Soc. Lecture Note Ser. 52, Cambridge Univ. Press, Cambridge, 80--102.

\bibitem{Bol_small} B.~Bollob\'{a}s, {\it Threshold functions for small subgraphs}, Math. Proc. Camb. Phil. Soc., 1981, {\bf 90}:~197--206.

\bibitem{BolWier} B.~Bollob\'{a}s, J.C. Wierman, {\it Subgraph counts and containment probabilities of balanced and unbalanced subgraphs in a large random graph}, Annals of the New York Academy of Sciences, 1989, {\bf 576}:~63--70.

\bibitem{Logic2} H.-D. Ebbinghaus, J. Flum, \emph{Finite model theory}, Perspectives in Mathematical Logic. Springer-Verlag, Berlin, second edition, 1999.

\bibitem{Ehren} A. Ehrenfeucht, \emph{An application of games to the completness problem for formalized theories}, Warszawa, Fund. Math., 1960, {\bf 49}:~121--149.

\bibitem{Erdos} P. Erd\H{o}s, A. R\'{e}nyi, \emph{On the evolution of random graphs}, Magyar Tud. Akad. Mat. Kutat\'{o} Int. K\"{o}zl., 1960, {\bf 5}: 17--61.

\bibitem{Muller} P.~Heinig, T.~M\"uller, M.~Noy, A.~Taraz, {\it Logical limit laws for minor-closed classes of graphs}, J. Comb. Theory, Ser. B, 2018, {\bf 130}: 158--206.

\bibitem{Libkin} L.~Libkin, {\it Elements of finite model theory},  Texts in Theoretical Computer Science. An EATCS Series, Springer-Verlag Berlin Heidelberg, 2004.

\bibitem{Janson} S. Janson, T. Luczak, A. Rucinski, {\it Random Graphs}, New
York, Wiley, 2000.

\bibitem{Shelah} S. Shelah, J.H. Spencer, \emph{Zero-one laws for sparse random graphs}, J. Amer. Math. Soc., 1988, \textbf{1}:97--115.

\bibitem{Spencer_ext} J.H.~Spencer, {\it Threshold functions for extension statements}, J. of Comb. Th., Ser A, 1990, {\bf 53}:~286--305.

\bibitem{Spencer_inf} J.H.~Spencer, {\it Infinite spectra in the first order theory of graphs}, Combinatorica, 1990, {\bf 10}(1):~95--102.

\bibitem{Strange} J.H.~Spencer, {\it The Strange Logic of Random Graphs}, Springer Verlag, 2001.

\bibitem{Tysk} J.~Tyszkiewicz, {\it On Asymptotic Probabilities of Monadic Second Order Properties},In: B\"orger E., J\"ager G., Kleine B\"uning H., Martini S., Richter M.M. (eds) Computer Science Logic. CSL 1992. Lecture Notes in Computer Science, 1993, 702:~425--439. Springer, Berlin, Heidelberg


\bibitem{VerbZhuk} O. Verbitsky, M. Zhukovskii M. (2017) The Descriptive Complexity of Subgraph Isomorphism Without Numerics. In: Weil P. (eds) Computer Science -- Theory and Applications. CSR 2017. Lecture Notes in Computer Science, vol 10304. Springer, Cham.

\bibitem{Veresh} N.K.~Vereshagin, A.~Shen, {\it Languages and calculus}, Moscow, MCCME, 2012.



\bibitem{Zhuk_neutral} M.E. Zhukovskii, {\it Estimation of the number of maximal extensions in the random graph}, Discrete Mathematics and Applications, 2012, 24(1): 79--107.

\bibitem{Zhuk_Logic} L.B. Ostrovsky, M.E. Zhukovskii, {\it Monadic second-order properties of very sparse random graphs}, Annals of pure and applied logic, 2017, {\bf 168}: 2087--2101.



\bibitem{Zhuk_inf} M.E. Zhukovskii, \emph{On infinite spectra of first order properties of random graphs}, Moscow Journal of Combinatorics and Number Theory, 2016, {\bf 6}(4): 73--102.

\bibitem{Zhuk_law} M.E. Zhukovskii, \emph{Zero-one $k$-law}, Discrete Mathematics, 2012, {\bf 312}: 1670--1688.



\bibitem{Survey} M.E.~Zhukovskii, A.M.~Raigorodskii, {\it Random graphs: models and asymptotic characteristics}, Russian Mathematical Surveys, {\bf 70}(1):~33--81, 2015.


\end{thebibliography}
\end{document}